\documentclass[onefignum,onetabnum]{siamonline190516}

\usepackage{natbib}
 \bibpunct[, ]{[}{]}{,}{n}{}{,}%

\def\real{\mathbb{R}}

\def\Mu{\mathcal{u}}
\renewcommand{\Pr}{\mathbb{P}}

\newcommand{\E}[1]{\mathbb{E}\left[#1\right]}
\newcommand{\Var}[1]{\mathrm{Var}\left[#1\right]}
\newcommand{\Cov}[1]{\mathrm{Cov}\left[#1\right]}

\def\Uj{U}
\def\Vj{V}

\usepackage{lipsum}
\usepackage{amsfonts}
\usepackage{graphicx}
\usepackage{epstopdf}
\usepackage{algorithmic}

\ifpdf
  \DeclareGraphicsExtensions{.eps,.pdf,.png,.jpg}
\else
  \DeclareGraphicsExtensions{.eps}
\fi
\usepackage{appendix}
\usepackage{subfigure}
\def\real{\mathbb{R}}

\def\Mu{\mathcal{u}}
\def\icc{\mathfrak{i}}
\def\ipow{\frac{2^{\icc+1}-1}{2^{\icc}}}
\def\nmB{n_m}
\def\nfB{n_f}
\renewcommand{\Pr}{\mathbb{P}}

\newcommand{\ebnote}[1]{\textsf{{\color{red}{ Houcine:}   #1 }}}

\newcommand{\revised}[1]{\textcolor{black}{#1}}

\usepackage{enumitem}
\setlist[enumerate]{leftmargin=.5in}
\setlist[itemize]{leftmargin=.5in}

\newsiamremark{remark}{Remark}
\newsiamremark{hypothesis}{Hypothesis}
\crefname{hypothesis}{Hypothesis}{Hypotheses}
\newsiamthm{claim}{Claim}
\newsiamthm{assumption}{Assumption}

\headers{A Stochastic Levenberg-Marquardt Method}{E. Bergou, Y. Diouane, V. Kungurtsev, and C. W. Royer}

\title{
A stochastic Levenberg-Marquardt method using random models with complexity results
}

\author{El Houcine Bergou\thanks{
  Mohammed VI Polytechnic University, Ben Guerir, Morocco (\email{elhoucine.bergou@um6p.ma}).
  }
\and Youssef Diouane \thanks{ISAE-SUPAERO, Universit\'e de Toulouse, 31055 Toulouse Cedex 4, France
  (\email{youssef.diouane@isae-supaero.fr}).}
\and Vyacheslav Kungurtsev \thanks{Department of Computer Science, Faculty of Electrical Engineering, Czech Technical University in Prague. Support for this author was provided by the OP VVV project
CZ.02.1.01/0.0/0.0/16\_019/0000765 ``Research Center for Informatics''  
  (\email{vyacheslav.kungurtsev@fel.cvut.cz}).}
  \and Cl\'ement W. Royer \thanks{LAMSADE, CNRS, Universit\'e Paris Dauphine-PSL, 75016 Paris, France. Support for this author was provided by Subcontract 3F-30222 from Argonne National Laboratory, by CNRS INS2I under the grant GASCON and by ANR through a PRAIRIE chair endowment (\email{clement.royer@dauphine.psl.eu}).}}

\usepackage{amsopn}

\begin{document}

\def\proofsCV{0}

\maketitle
\footnotesep=0.4cm

{\small
\begin{abstract} 
Globally convergent variants of the Gauss-Newton algorithm are often the 
methods of choice to tackle nonlinear least-squares problems. Among such 
frameworks, Levenberg-Marquardt and trust-region methods are two 
well-established, similar paradigms. Both schemes have been 
studied when the Gauss-Newton model is replaced by a random model that is 
only accurate with a given probability. 
Trust-region schemes have also been applied to problems where the 
objective value is subject to noise: this setting is of particular interest 
in fields such as data assimilation, where efficient methods that can adapt 
to noise are needed to account for the intrinsic uncertainty in the input
data.\\
In this paper, we describe a stochastic Levenberg-Marquardt algorithm 
that handles noisy objective function values and random models, 
provided sufficient accuracy is achieved in probability. Our method relies 
on a specific scaling of the regularization parameter, that allows us to 
leverage existing results for trust-region algorithms. Moreover, we exploit 
the structure of our objective through the use of a family of  
stationarity criteria tailored to least-squares problems. Provided
the probability of accurate function estimates and models is sufficiently 
large, we bound the expected number of iterations needed to 
reach an approximate stationary point, which generalizes results based on 
using deterministic models or noiseless function values.  \revised{We illustrate 
the link between our approach and several applications related to inverse problems and machine learning.}
\end{abstract}



\section{Introduction}
\label{sec:intro}

Minimizing a nonlinear least-squares function is one of the most classical 
problems in numerical optimization, arising in a variety of applications. In
numerous cases, the objective function to be optimized can only be 
accessed through noisy estimates. Typical occurrences of such a formulation 
can be encountered while solving inverse 
problems~\cite{PCourtier_JNThyepaut_AHollingsworth_1994,
ATarantola_2005,YTremolet_2007} or while minimizing the error 
of a model in the context of machine 
learning~\cite{LBottou_FECurtis_JNocedal_2018}. In such cases, the 
presence of noise is often due to the estimation of the objective function 
via cheaper, less accurate calculations. For instance, in data-fitting 
problems, part of the data is often left aside while computing the function or 
derivative estimates, due to the cost of considering the entire available 
dataset. 

Such concerns have motivated the development of optimization frameworks 
that cope with inexactness in the objective function or its derivatives. In 
particular, the field of derivative-free 
optimization~\cite{ARConn_KScheinberg_LNVicente_2009}, where it is assumed 
that the derivatives exist but are unavailable for use in an algorithm, 
has expanded in recent years with the introduction of random 
models~\cite{JLarson_MMenickelly_SMWild_2019}. In a seminal work, Bandeira 
et al.~\cite{ASBandeira_KScheinberg_LNVicente_2012} applied arguments from 
compressed sensing to guarantee accuracy of quadratic models whenever the 
Hessian exhibits a certain (unknown) sparsity pattern. 
Trust-region methods based on general probabilistic models were then proposed 
and convergence to first- and second-order stationary points was established 
under appropriate accuracy assumptions on the 
models~\cite{ASBandeira_KScheinberg_LNVicente_2014}. Global convergence rates 
were derived for this approach in expectation and with high 
probability~\cite{SGratton_CWRoyer_LNVicente_ZZhang_2018}. In a parallel line 
of work, trust-region methods with probabilistic models were extended so as to 
accommodate noisy function values by Blanchet et 
al~\cite{RChen_MMenickelly_KScheinberg_2018}. To this end, the analysis 
accounts for two sources of randomness, arising from both the noisy function 
estimates and the random construction of the models. It can then be shown 
that the trust-region scheme requires at most $\mathcal{O}(\epsilon^{-2})$ 
iterations in expectation to drive the gradient norm below some threshold 
$\epsilon$~\cite{JBlanchet_CCartis_MMenickelly_KScheinberg_2019}.

In the context of derivative-free least-squares problems with exact function 
values, various deterministic approaches based on globalization of the 
Gauss-Newton method have been studied. The algorithms developed in the 
derivative-free community are mostly of trust-region type, and rely on 
building models that satisfy the so-called fully linear property, which 
requires the introduction of a so-called criticality step to guarantee
its satisfaction throughout the algorithmic process~\cite{CCartis_LRoberts_2019,
HZhang_ARConn_2012,HZhang_ARConn_KScheinberg_2010,SMWild_2017}.
The recent DFO-GN algorithm~\cite{CCartis_LRoberts_2019} was equipped with a 
complexity result with a bound of the same order as derivative-free 
trust-region methods for generic functions~\cite{RGarmanjani_DJudice_LNVicente_2016}. 
As for general problems, random models emerged as a way of 
relaxing the need for accuracy at every iteration. A Levenberg-Marquardt 
algorithm based in this idea was proposed by Bergou et 
al~\cite{EBergou_SGratton_LNVicente_2016}, motivated by problems from data 
assimilation: this method extends the classical Levenberg-Marquardt scheme by 
replacing the gradient of the objective function with a noisy estimate that is 
only accurate in probability. Using reasoning similar to the trust-region 
case~\cite{ASBandeira_KScheinberg_LNVicente_2014}, almost-sure global 
convergence to a first-order stationary point was established.

The case of noisy least squares has also been examined. A recent 
preprint~\cite{CCartis_JFiala_BMarteau_LRoberts_2019} proposed an efficient 
approach for handling noisy values in practice, but did not provide 
theoretical guarantees. A Levenberg-Marquardt framework for noisy 
optimization without derivatives was proposed by Bellavia et 
al.~\cite{SBellavia_SGratton_ERiccietti_2018}. This method assumes that 
function values can be estimated to a prescribed accuracy level 
and explicitly maintains a sequence of these levels throughout the 
iterations of the algorithm. Since the noise level must be small compared to the 
norm of the Levenberg-Marquardt step, one must be able to 
reduce the noise level when necessary (note that this idea resembles the 
criticality step of derivative-free model-based methods). In certain 
applications, this may be deemed as too expensive. By contrast, the use of 
random models and estimates with properties only guaranteed in probability 
allows for arbitrarily bad estimates, which seems more economical at the 
iteration level, and does not exclude the possibility of computing 
good steps from bad, cheap estimates. Probabilistic properties thus 
represent a valuable alternative to the above approach. Furthermore, the 
connections between Levenberg-Marquardt and trust-region 
methods~\cite{JJMore_1977} suggest that the analysis of the latter on 
noisy problems can help with studying the former.

In this paper, we propose a stochastic Levenberg-Marquardt framework that 
builds upon the algorithm of Bergou et 
al.~\cite{EBergou_SGratton_LNVicente_2016} to handle both random models and 
noise in the function evaluations. Our setup allows for 
arbitrarily inaccurate models or function estimates: provided those occur at 
the same time with a small probability, we can equip our method with 
complexity guarantees. Our analysis adapts that of the stochastic 
trust-region framework using random models proposed 
in~\cite{JBlanchet_CCartis_MMenickelly_KScheinberg_2019,
RChen_MMenickelly_KScheinberg_2018}, thanks to an appropriate definition 
of the Levenberg-Marquardt regularization parameter. In addition, we 
quantify convergence using a scaled stationarity criterion that accounts for 
the least-squares structure of our problem and covers standard as well as 
recently proposed metrics~\cite{CCartis_NIMGould_PhLToint_2013d}.

The remainder of the paper is organized as follows. Section~\ref{sec:LM} 
describes our Levenberg-Marquardt algorithm. Section~\ref{sec:LMprobaprop} 
details accuracy requirements that we enforce for the noisy function values and 
the probabilistic models. Worst-case 
guarantees for our framework are provided in Section~\ref{sec:theory}.
\revised{Section~\ref{sec:applis} discusses several applications of our method 
to inverse problems and machine learning.} 
Section~\ref{sec:conclusion} concludes our work.

\section{A Levenberg-Marquardt algorithm based on estimated values} 
\label{sec:LM}

This paper is concerned with the following nonlinear least squares problem:
\begin{equation} \label{eq:mainpb}
	\min_{x \in \real^n} \; f(x)  := \frac{1}{2}\|r(x)\|^2,
\end{equation}
where $r:\real^n \rightarrow \real^{\ell}$ is a so-called residual, vector-valued 
function, which we assume to be continuously differentiable, and $\|\cdot\|$ is 
the Euclidean norm. We consider that $r$ and its derivatives cannot be 
accessed directly for algorithmic purposes. Therefore, we will present an 
algorithm that relies on random approximations of these quantities that are 
of good quality with a certain probability.

In the rest of this section, we recall the main features of the 
Levenberg-Marquardt method, then describe our extension of this algorithm to 
handle inexact function and derivative values.

\subsection{Deterministic Levenberg-Marquardt paradigm}
\label{subsec:determLM}

Popular approaches for solving problem~\eqref{eq:mainpb} are based on the 
Gauss-Newton model. Given a current iterate~$x_j$, a step is computed as a 
solution of the linearized least-squares subproblem
\[
\min_{s \in \real^n} \; \revised{\frac{1}{2}\|r(x_j)+ J(x_j) s\|^2},
\]
where \revised{$J(\cdot)$ denotes the Jacobian of $r$.}
This subproblem possesses a unique solution if \revised{$J(x_j)$} has full 
column rank, and in that case the step is a descent direction for~$f$. When 
\revised{$J(x_j)$} is not of full 
column rank, the introduction of a regularization parameter can lead to 
similar properties. This is the underlying idea behind the 
Levenberg-Marquardt algorithm~\cite{KLevenberg_1944,DMarquardt_1963,
MROsborne_1976}, a globally convergent method based upon the Gauss-Newton 
model. At each iteration, one considers a step of the form
\revised{$- ( J(x_j)^\top J(x_j) + \gamma_j I )^{-1} J(x_j)^\top r(x_j)$},
corresponding to the unique solution of
\begin{equation} \label{eq:subproblemLM}
\min_{s \in \real^n} \; \revised{\frac{1}{2}\|r(x_j) + J(x_j) s\|^2 + 
\frac{1}{2}\gamma_j \|s\|^2,}
\end{equation}
where $\gamma_j \ge 0$ is an appropriately chosen regularization parameter, 
typically updated in the spirit of the classical trust-region radius update 
strategy at each iteration. In our proposed scheme, we draw a closer connection 
between this parameter and the trust-region radius by scaling~$\gamma_j$ 
by the norm of the gradient of the Gauss-Newton model. This approach has 
been previously proposed in Levenberg-Marquardt-type methods and leads to 
complexity guarantees that match those of trust-region 
schemes~\cite{EBergou_YDiouane_VKungurtsev_2020,RZhao_JFan_2016}.

\subsection{Algorithmic framework based on estimates}
\label{subsec:LMalgoestimates}

Our goal is to propose a method that applies to instances of~\eqref{eq:mainpb} 
for which neither $r$ nor $J$ can be accessed directly. Consequently, we 
consider a variant of the Levenberg-Marquardt algorithm, described in 
Algorithm~\ref{alg:LM}, in which both the 
function and gradient values are approximated.
At every iteration, estimates of the values of $f$ and its derivative at the 
current iterate are computed, and used to compute a Gauss-Newton type 
model~\eqref{eq:LMmodelalgo}. A regularized version of this model is then 
approximately minimized, yielding a trial step $s_j$. New estimates of the 
objective at the current and trial point are computed: this new point is  
accepted if the ratio $\rho_j$ between the estimated function decrease and 
the model decrease is sufficiently large.

A key feature of our method is that the regularization parameter is defined 
using a specific scaling formula: namely, we set
$\gamma_j = \mu_j \| J_{m_j}^\top r_{m_j} \|$ where $\mu_j \ge 0$. 
The parameter $\mu_j$ is updated depending on the value of $\rho_j$, and also 
on a condition involving the model gradient. Such 
updates are typical of derivative-free model-based methods based on 
random estimates~\cite{ASBandeira_KScheinberg_LNVicente_2014,
EBergou_SGratton_LNVicente_2016,RChen_MMenickelly_KScheinberg_2018,
SGratton_CWRoyer_LNVicente_ZZhang_2018}. \revised{Note that we follow these 
earlier references in checking the condition 
$\|J_{m_j}^\top r_{m_j}\| \ge \tfrac{\eta_2}{\mu_j}$ at the end of the 
iteration, while a more practical application would evaluate it at Step 2 
of our algorithm. Our theoretical analysis remains unchanged.}

\begin{algorithm}
{\bf \bf A Levenberg-Marquardt method using random models and estimates.}
\vspace{-2ex}
\label{alg:LM}
\begin{rm}
\begin{description}
\item[]
\item[Initialization] \ \\
Define $\eta_1 \in (0,1)$, $\eta_2,\mu_{\min}>0$, and $\lambda>1$. 
	Choose $x_0$ and $\mu_0 \geq \mu_{\min}$.
\vspace{1ex}
\item[For $j=0,1,2,\ldots$] \ \\
\vspace{-2ex}
		\begin{enumerate}
			\item Compute an estimate $f^0_j=\tfrac{1}{2}\|r_j^0\|^2$ of $f(x_j)$.
			\item Compute $r_{m_j}$ and $J_{m_j}$, the residual and the Jacobian 
			estimate at $x_j$, set\\ 
			$\gamma_j=\mu_j\|J_{m_j}^\top r_{m_j}\|$, and define the model $m_j$ of 
			$f$ around $x_j$ by:
				\begin{eqnarray} \label{eq:LMmodelalgo}
					\forall s \in \real^n,\ m_j(x_j+s) &:= 
					&\frac{1}{2}\|r_{m_j}+J_{m_j}s\|^2 \\
					&= &\frac{1}{2}\|r_{m_j}\|^2 + (J_{m_j}^\top r_{m_j})^\top s + 
					\frac{1}{2}s^\top J_{m_j}^\top J_{m_j}s. \nonumber
				\end{eqnarray}
			\item Compute an approximate solution $s_j$ of the subproblem
				\begin{equation} \label{eq:LMsubproblem}
					\min_{s \in \real^n} m_j(x_j+s) + \frac{\gamma_j}{2}\|s_j\|^2.\;
				\end{equation}
			\item Compute an estimate $f^s_j=\tfrac{1}{2}\|r_j^s\|^2$ of $f(x_j+s_j)$, 
			then compute 
				\begin{equation*}
					\rho_j \; := \;    
					\frac{f_j^0 - f_j^s}{m_j(x_{j}) - m_j(x_j+s_j) 
					- \tfrac{\gamma_j}{2}\|s_j\|^2}.\;
				\end{equation*}
			\item If $\rho_j \geq \eta_{1}$ and 
			$\|J_{m_j}^\top r_{m_j}\| \ge \tfrac{\eta_2}{\mu_j}$, 
			set $x_{j+1}=x_j+s_j$ and 
			$\mu_{j+1} = \max\left\{\tfrac{\mu_j}{\lambda},\mu_{\min}\right\}$.\\
			Otherwise, set $x_{j+1}=x_j$ and $\mu_{j+1} = \lambda\,\mu_j$.
		\end{enumerate}
\end{description}
\end{rm}
\end{algorithm}

\section{Probabilistic properties of models and function estimates} 
\label{sec:LMprobaprop}

The framework of Algorithm~\ref{alg:LM} allows for approximations of the 
objective function and its derivatives to construct both the models and 
estimate the function values. In this section, we consider that the function 
values and the derivatives can only be accessed through noisy approximations, 
and we define accuracy formulas in a deterministic and probabilistic sense.

\subsection{Deterministic accuracy}
\label{subsec:determaccuracy}

We begin by describing our accuracy requirements for models of the form given 
in~\eqref{eq:LMmodelalgo}. Following previous work on derivative-free 
Levenberg-Marquardt methods~\cite{EBergou_SGratton_LNVicente_2016}, we 
propose the following accuracy definition, and motivate its use further below.

\begin{definition} \label{defi:firstordacc}
	Consider a realization of Algorithm~\ref{alg:LM}, and the model $m_j$ 
	of $f$ defined around the iterate $x_j$ of the 
	form~\eqref{eq:LMmodelalgo}, and let $\kappa_{ef},\kappa_{eg}>0$. Then, 
	the model $m_j$ is called \emph{$(\kappa_{ef},\kappa_{eg})$-first-order 
	accurate} with respect to $(x_j,\mu_j)$ if the following properties hold:
	\begin{eqnarray}
		\label{eq:firstordaccgrad}
		\| J_{m_j}^\top r_{m_j} - \revised{J(x_{j})^\top r(x_j)}\| 
		&\le &\frac{\kappa_{eg}}{\mu_j},\\
		\label{eq:firstordaccfunc}
		\left| \|r(x_j)\|^2 - \|r_{m_j}\|^2 \right| 
		&\le &\frac{2\kappa_{ef}}{\mu_j^2}.
	\end{eqnarray}
\end{definition}

\begin{remark}
Definition~\ref{defi:firstordacc} resembles that of fully linear 
models in derivative-free optimization~\cite{ARConn_KScheinberg_LNVicente_2009}, 
thanks to a specific scaling of the regularization parameter. In particular, 
the accuracy requirement for the model gradient~\eqref{eq:firstordaccgrad} differs 
from the first-order accuracy property introduced by Bergou, Gratton and 
Vicente~\cite{EBergou_SGratton_LNVicente_2016}. With our choice of notation, 
the latter corresponds to:
\[
		\left \| J_{m_j}^\top r_{m_j} - 
		\revised{J(x_j)^\top r(x_j)} \right \| \; \le \; 
		\frac{\kappa_{eg}}{\gamma_j},
\]
One thus sees that this property uses $\gamma_j=\mu_j\|J_{m_j}^\top r_{m_j}\|$ 
in the right-hand side, while ours~\eqref{eq:firstordaccgrad} uses $\mu_j$. 
The purpose of our new property is twofold. First, it allows us to measure the 
accuracy in formulas~\eqref{eq:firstordaccgrad} 
and~\eqref{eq:firstordaccfunc} through a parameter that is updated in an 
explicit fashion throughout the algorithmic run: this is a key property for 
performing a probabilistic analysis of optimization methods. Secondly, we 
believe this choice to be a better reflection of the relationship between the 
Levenberg-Marquardt and the trust-region parameter. Indeed, the global solution 
of the subproblem~\eqref{eq:LMsubproblem} is given by 
\revised{$d_j = - \left(J(x_j)^\top J(x_j) + \gamma_j I\right)^{-1} J(x_j)^\top r(x_j)$}, 
which is also the solution of the trust-region subproblem
\begin{equation}
	\left\{
		\begin{array}{ll}
			\min_d &\frac{1}{2}\|\revised{r(x_j) + J(x_j)} d \|^2 \\
			s.t.		&\|d\| \le \delta_j = \|d_j\|.
		\end{array}
	\right.
\end{equation}
As a result, we see that for a large value of $\gamma_j$, one would have
$\delta_j =\mathcal{O}\left( \frac{\|\revised{J(x_j)^\top r(x_j)}\|}{\gamma_j}\right)$,
which suggests that $\gamma_j$ is not exactly equivalent to the inverse of the 
trust-region radius (as used in earlier 
work~\cite{EBergou_SGratton_LNVicente_2016}), \revised{but rather} is an 
equivalent to $\tfrac{\|\revised{J(x_j)^\top r(x_j)}\|}{\delta_j}$. As a 
result, the parameter $\mu_j$ can be thought as equivalent to 
$\tfrac{1}{\delta_j}$: our property \eqref{eq:firstordaccgrad} thus
matches the gradient accuracy condition in fully linear 
models~\cite{ARConn_KScheinberg_LNVicente_2009}.
\end{remark}

Note that Definition~\ref{defi:firstordacc} contains two conditions related 
to the model Jacobian and the value at the current point. The latter property, 
described by~\eqref{eq:firstordaccfunc}, is necessary because our method 
relies on inexact residual values. For the same reasons, we define accuracy 
conditions for the estimates computed at every iteration of our method. 

\begin{definition} \label{defi:accfunc}
	\revised{Consider a realization of Algorithm~\ref{alg:LM}, and the 
	residual estimates $r^0_j$ and $r^s_j$ computed at iteration $j$.}
	Given $\varepsilon_f > 0$, we say that $r^0_j$ and $r^s_j$
	are \emph{$\varepsilon_f$-accurate estimates of $f(x_j)$ and 
	$f(x_j+s_j)$} \revised{with respect to $(x_j,\mu_j)$} if
	\begin{equation} \label{eq:accfunc}
		\left| \|r^0_j\|^2 - \|r(x_j)\|^2 \right| 
		\le \frac{2\varepsilon_f}{\mu_j^2} 
		\quad \mathrm{and} \quad
		\left| \|r^s_j\|^2 - \|r(x_j+s_j)\|^2 \right| 
		\le \frac{2\varepsilon_f}{\mu_j^2}.
	\end{equation}
\end{definition}

Here again, we point out that the parameter $\mu_j$ plays the role of a 
reciprocal of the trust-region radius. In that sense, the previous definitions 
are consistent with the definitions of sufficient accuracy presented in the 
case of stochastic trust-region 
methods~\cite{RChen_MMenickelly_KScheinberg_2018}.

\subsection{Probabilistic properties}
\label{subsec:probaaccuracy}

The deterministic properties described in the previous section allow for a 
deterministic inexact analysis of Algorithm~\ref{alg:LM}. We are further 
interested in the case where the models and the estimates are computed in a 
stochastic fashion. This introduction of randomness implies that the 
iterates, regularization parameters and trial steps become stochastic 
processes. We will thus denote by $X_j$, $\Gamma_j$, $\Mu_j$ and $S_j$ 
the random quantities at iteration $j$; the notations $x_j= X_j(\omega)$, 
$\gamma_j= \Gamma_j(\omega)$, $\mu_j= \Mu_j(\omega)$ and $s_j = S_j(\omega)$ 
correspond to realizations of these processes.
The random model at iteration $j$ of Algorithm~\ref{alg:LM} will be denoted 
by $M_j$, and we use $m_j=M_j(\omega)$ for a realization of that model 
(corresponding to a realization of the algorithm). Similarly, we let 
$r_{M_j}$ and $J_{M_j}$ denote the estimates of the residual $r(X_j)$ and the 
Jacobian $J(X_j)$ at iteration $j$, with their realizations denoted by
$r_{m_j} = r_{M_j} (\omega)$, and $J_{m_j} = J_{M_j} (\omega)$.
We also define $R_j^0$ and $R_j^s$ as the random estimates of $r(X_j)$ 
and $r(X_j+S_j)$. The realizations of $R_j^0$ and $R_j^s$ will be denoted by 
$r_j^0$ and $r_j^s$.

Finally, we give the probabilistic equivalents of 
Definitions~\ref{defi:firstordacc} and~\ref{defi:accfunc} below.

\begin{definition} \label{defi:probafirstordacc}
	Let $p \in (0,1]$, $\kappa_{ef} >0$ and $\kappa_{eg} >0$. A sequence of 
	random models $\{M_j\}$ is said to be \emph{$p$-probabilistically 
	$\{\kappa_{ef},\kappa_{eg}\}$-first-order accurate} with respect to the 
	sequence $\{X_j,\Mu_j\}_j$ if the events
	\[
		\Uj_j \; := \; 
		\left\{\,\left\| J_{M_j}^\top r_{M_j}- J(X_j)^\top r(X_j) \right\| 
		\; \le \; \frac{\kappa_{eg}}{\Mu_j}\ \&\ 
		\left| \|r(X_j)\|^2 - \|r_{M_j}\|^2 \right| \; \le \; 
		\frac{2\kappa_{ef}}{\Mu_j^2}
		\right\}
	\]
	satisfy the following condition
	\begin{equation} \label{pjcondition}
		p_j^* \;  := \;
		 P(U_j | \mathcal{F}^{M \cdot R}_{j-1}) \; \geq \; p,
	\end{equation}
	where $\mathcal{F}^{M\cdot R}_j=\sigma(M_0,\ldots,M_{j-1},
	R^0_0,R^s_0,\dots,R^0_{j-1},R^s_{j-1})$ is the 
	$\sigma$-algebra generated by $M_0,\ldots,M_{j-1}$ and 
	$R^0_0,R^s_0,\dots,R^0_{j-1},R^s_{j-1}$.
\end{definition}

\begin{definition} \label{defi:probaccfun}
	Given constants  $\varepsilon_f>0$, and $q \in (0,1]$, the sequences 
	of random quantities $R_j^0$ and $R_j^s$ is called 
	$q$-probabilistically $\varepsilon_{f}$-accurate, for corresponding 
	sequence $\{ X_j, \Mu_j\}_j$, if the events
	\[
		\Vj_j \; := \; 
		\left\{ \left|\|R_j^0\|^2 - \|R(X_j)\|^2 \right| \; \le \;
		\frac{2\varepsilon_f}{\Mu_j^{2}}  
		\quad \mathrm{and} \quad 
		\left|\|R_j^s\| - \|R(X_j+S_j)\|^2 \right| \; \le \;
		\frac{2\varepsilon_f}{\Mu_j^{2}} \right \}
	\]
	satisfy the following condition
	\begin{equation} \label{qjcondition}
		q_j^* \;  := \; P(\Vj_j | \mathcal{F}^{M \cdot R}_{j-1/2}) \; \geq \; q,
	\end{equation}
	where $\mathcal{F}^{M\cdot R}_{j-1/2}$ is the $\sigma$-algebra generated by 
	$M_0,\ldots,M_j, R^0_0,R^s_0 \ldots, R_{j-1}^0,R_{j-1}^s$.
\end{definition}

\section{Convergence rate analysis} \label{sec:theory} 

In this section, we provide a theoretical study of our algorithm using 
stochastic process theory. Our methodology follows the approach by Blanchet et 
al.~\cite{JBlanchet_CCartis_MMenickelly_KScheinberg_2019} for trust-region 
methods.\footnote{When the \revised{original, unpublished} version of this 
paper~\cite{EBergou_YDiouane_VKungurtsev_CWRoyer_2018} was released, our 
complexity results differed from those that Blanchet et 
al.~\cite{JBlanchet_CCartis_MMenickelly_KScheinberg_2019} had obtained at that 
time for their algorithm. The analysis was then improved in the final, published 
version~\cite{JBlanchet_CCartis_MMenickelly_KScheinberg_2019} and matches the 
the results of \revised{this earlier, unpublished 
version}~\cite{EBergou_YDiouane_VKungurtsev_CWRoyer_2018}.} 
However, our setup introduces a number of variations that require us to 
make some modifications in key components of the analysis. In particular, we consider a 
measure of stationarity that exploits the least-squares form of the 
problem. Indeed, in order to take advantage of the least-squares structure 
of our problem, we focus on a scaled optimality criterion inspired by 
previous proposals for least-squares 
problems~\cite{CCartis_NIMGould_PhLToint_2013d}. Rather than considering 
$\|\nabla f(x)\| \le \epsilon$ as our stationarity condition, we introduce 
the criterion:
\begin{equation} \label{eq:wcccrit}
	\|r(x)\| \le \epsilon_p \quad \mathrm{or} \quad 
	\|g_r^{\icc}(x)\| \le \epsilon_d,
\end{equation}
where $g_r^{\icc}$ is the so-called scaled gradient defined for a fixed 
integer $\icc \in \mathbb{N} \cup \{-1\}$ by
\begin{equation} \label{eq:scaledgrad}
	g_r^{\icc}(x) \; := \; \left\{
		\begin{array}{ll}
			\tfrac{\|J(x)^T r(x)\|}{\|r(x)\|^{\ipow}} 
			&\mathrm{if\ } \|r(x)\| \neq 0, \\
			 & \\
			0 &\mathrm{otherwise.}
		\end{array}
		\right.
\end{equation}
In our framework, the tolerance $\epsilon_p$ corresponds to a tolerance after 
which the noise from the estimated values would dominate the actual residual 
value in the case of small or zero residuals. When the residuals at the 
optimum are non-zero, however, we consider a scaled version of the optimality 
conditions, captured by the scaled gradient $g_r^{\icc}$. In that case, the 
tolerance $\epsilon_d$ can be seen as a scaled version of the classical 
gradient tolerance.
Note that our definition of the scaled gradient~\eqref{eq:scaledgrad} matches 
previous proposals~\cite{CCartis_NIMGould_PhLToint_2013d,NIMGould_TRees_JAScott_2019} 
for $\icc=0$, while it corresponds to the classical gradient for $\icc=-1$. 
As $\icc \rightarrow \infty$, we have $g_r^\icc (x) \rightarrow \tfrac{\|J(x)^T r(x)\|}{\|r(x)\|^2}$ 
(when $\|r(x)\|\neq 0$), and the condition $\tfrac{\|J(x)^T r(x)\|}{\|r(x)\|^2} \ge \epsilon_d$ 
resembles that of gradient dominance of degree 1 
(see~\cite{CCartis_NIMGould_PhLToint_2013d}).

In order to connect the variation in the objective value with the scaled 
gradient criterion, we will rely on the \revised{lemma below, that is a direct 
corollary of Gould et al.~\cite[Lemma 3.11]{NIMGould_TRees_JAScott_2019} with 
fixed $\icc$ (we refer the reader to this earlier 
result~\cite[Lemma 3.11]{NIMGould_TRees_JAScott_2019} for a full proof).}

%
\begin{lemma}
\label{lemma:oursab}
	For any $a>b \ge 0$ and $c \in \real$,
	\begin{equation} \label{eq:oursab}
		a^2 - b^2 \le c \quad \Rightarrow \quad 
		a^{1/2^{\icc}} - b^{1/2^{\icc}} \le \frac{c}{a^{\ipow}}.
	\end{equation}
\end{lemma}

\subsection{Assumptions and deterministic results} \label{subsec:assumdeterm}

The objective function will be required to satisfy the following 
\revised{assumptions}.

\begin{assumption}
\label{assum:fC11}
	$f$ is continuously differentiable on an open set containing the level 
	set\\ $\mathcal{L}(x_0) = \left\{ x \in \real^n | f(x) \le f(x_0)\right\}$, 
	with Lipschitz continuous gradient, of Lipschitz constant $\nu$.
\end{assumption}

We also require that the Jacobian model is uniformly bounded over the sequence 
of iterates, for every realization of the algorithm. 

\begin{assumption}
\label{assum:Jacmodel}
	There exists $\kappa_{J_m}>0$ such that for all $j$ and all realizations 
	$J_{m_{j}}$ of the $j$-th model Jacobian, one has:
	\[
		\| J_{m_j} \| \leq \kappa_{J_m}.
	\]
\end{assumption}

Finally, we state the assumptions on the approximate solve of the 
subproblem~\eqref{eq:subproblemLM}.  \revised{The first assumption states that the trial 
step achieves a fraction of Cauchy decrease for the regularized model.}

\begin{assumption}
\label{assum:modeldecrease}
	There exists $\theta_{fcd} > 0$ such that for every iteration $j$ of any 
	realization of the algorithm,
	\begin{equation}
	\label{eq:modeldecrease}
		m_j(x_j) - m_j(x_j+s_j)-\frac{\gamma_j}{2}\|s_j\|^2 \; \geq \; 
		\frac{\theta_{fcd}}{2}\,\frac{\|J_{m_j}^\top r_{m_j}\|^2}{\|J_{m_j}\|^2+\gamma_j}.
	\end{equation}
\end{assumption}

\revised{
The second assumption states that the trial step satisfies desirable bounds 
on its norm and part of the model decrease, two terms that arise naturally 
in the theoretical analysis (see the proof of Lemma~\ref{lemma:condmodelgraditsucc}).}

\begin{assumption}
\label{assum:stepbounds}
	At each iteration $j$ and for every realization of the algorithm, the step 
	size satisfies
	\begin{equation}
	\label{eq:stepsizebound}
		\| s_j \| \; \leq \; \frac{2\|J_{m_j}^\top r_{m_j}\|}{\gamma_j} 
		= \frac{2}{\mu_j},
	\end{equation}
	and there exists $\theta_{in} > 0$ such that
	\begin{equation}
	\label{eq:stepproductbound}
		| s_j^{\top}\,(\gamma_j\,s_j+J_{m_j}^\top r_{m_j}) | \; \leq \;
		\frac{4\,\|J_{m_j}\|^2\,\|J_{m_j}^\top r_{m_j}\|^2 + 
		2\theta_{in}\,\|J_{m_j}^\top r_{m_j}\|^2}
		{\gamma_j^2} \; = \; 
		\frac{4 \|J_{m_j}\|^2+ 2 \theta_{in}}{\mu_j^2}.
	\end{equation}
\end{assumption}

Several choices for the approximate minimization of $m_j(x_j+s)$ verify 
relations~\eqref{eq:modeldecrease}, \eqref{eq:stepsizebound} 
and~\eqref{eq:stepproductbound}. In particular, 
\revised{Assumptions~\ref{assum:modeldecrease} and~\ref{assum:stepbounds} 
hold for the exact minimizer of the quadratic subproblem (for any $\theta_{in}>0$), 
but they are also valid for inexact steps (for some $\theta_{in}>0$) such as 
the Cauchy step or a step computed by the truncated Conjugate Gradient 
algorithm~\cite[Lemma 5.1]{EBergou_SGratton_LNVicente_2016}.}

As shown by the lemma below, our assumptions guarantee that an 
accurate model also provides an accurate estimate for the trial 
step.

\begin{lemma} \label{lemma:accmodelnextstep}
	Let Assumptions~\ref{assum:fC11}, \ref{assum:Jacmodel}, and 
	\ref{assum:stepbounds} hold for a realization of Algorithm~\ref{alg:LM}. 
	Consider the $j$-th iteration of that realization, and suppose that 
	$m_j$ is $(\kappa_{ef},\kappa_{eg})$-first-order accurate. Then,
	\begin{equation} \label{eq:accmodelnextstep}
		\left| f(x_j+s_j) - m_j(x_j+s_j) \right| \; \le \;
		\frac{\kappa_{efs}}{\mu_j^2},
	\end{equation}
	where $\kappa_{efs}  := \frac{2 \kappa_{ef}+4\kappa_{eg}+4(\nu+\kappa_{J_m}^2)}{2}$.
\end{lemma}

\begin{proof}
	Using a Taylor expansion of the function $f$ around $x_j$ and the definition 
	of $m_j$, we have:
	\begin{eqnarray*}
		\left| f(x_j+s_j) - m_j(x_j+s_j) \right| &\le 
		&\left| f(x_j)+(J(x_j)^\top r(x_j))^\top s_j - m_j(x_j+s_j) \right| + 
		\frac{\nu}{2}\|s_j\|^2 \\
		&\le &\left| f(x_j) - m_j(x_j) \right| + 
		\left| \left(J(x_j)^\top r(x_j)-J_{m_j}^\top r_{m_j}\right)^\top s_j \right| 
		+ \frac{\|J_{m_j}^\top J_{m_j}\|+\nu}{2}\|s_j\|^2.
	\end{eqnarray*}
	By Assumptions~\ref{assum:Jacmodel} and \ref{assum:stepbounds}, 
	this leads to
	\begin{equation*}
		\left| f(x_j+s_j) - m_j(x_j+s_j) \right| \le 
		\frac{\kappa_{ef}}{\mu_j^2} + \frac{\kappa_{eg}}{\mu_j}\|s_j\| 
		+ \frac{\kappa_{J_m}^2+\nu}{2}\|s_j\|^2 
		\le \frac{\kappa_{ef}}{\mu_j^2}+ \frac{2\kappa_{eg}}{\mu_j^2} 
		+ \frac{2(\kappa_{J_m}^2+\nu)}{\mu_j^2},
	\end{equation*}
	hence the result.
\end{proof}

The next lemmas describe useful results that hold for any 
realization of Algorithm~\ref{alg:LM}: they will be instrumental 
in studying the behavior of the method in a probabilistic setting 
(see Section~\ref{subsec:keythm}).


\begin{lemma} \label{lemma:decreasegoodmodel}
	Let Assumptions~\ref{assum:fC11}, \ref{assum:Jacmodel}, 
	\ref{assum:modeldecrease}, and \ref{assum:stepbounds} hold for a 
	realization of Algorithm~\ref{alg:LM}, and consider its $j$-th iteration.
	If the model is $(\kappa_{ef},\kappa_{eg})$-first-order accurate and
	\begin{equation} \label{eq:mudecreasegoodmodel}
		\mu_j \ge  \max\left\{\kappa_{J_m}^2,
		\frac{8(\kappa_{ef}+\kappa_{efs})}{\eta_1 \theta_{fcd}}\right\}
		\frac{1}{\|J_{m_j}^\top r_{m_j}\|},
	\end{equation}
	then the trial step $s_j$ satisfies
	\begin{equation} \label{eq:decreasegoodmodel}
		\|r(x_j+s_j)\|^2 - \|r(x_j)\|^2 \; \le \; -\frac{\eta_1 \theta_{fcd}}{4}
		\frac{\|J_{m_j}^\top r_{m_j}\|}{\mu_j}.
	\end{equation}
\end{lemma}

\begin{proof}
	Since the model is $(\kappa_{ef},\kappa_{eg})$-first-order accurate, we have:
	\begin{eqnarray*}
		f(x_j+s_j) - f(x_j) &= 
		&f(x_j+s_j)-m(x_j+s_j) + m(x_j+s_j)-m_j(x_j) + m_j(x_j) - f(x_j) \\
		&\le &\frac{\kappa_{efs}}{\mu_j^2} + m(x_j+s_j)-m_j(x_j) 
		+ \frac{\kappa_{ef}}{\mu_j^2} \\
		&\le &\frac{\kappa_{ef}+\kappa_{efs}}{\mu_j^2} -\frac{\eta_1 \theta_{fcd}}{2}\,
		\frac{\|J_{m_j}^\top r_{m_j}\|^2}{\kappa_{J_m}^2+\gamma_j} 
		= \frac{\kappa_{ef}+\kappa_{efs}}{\mu_j^2} -\frac{\eta_1 \theta_{fcd}}{2}\,
		\frac{\|J_{m_j}^\top r_{m_j}\|^2}
		{\kappa_{J_m}^2+\mu_j\|J_{m_j}^\top r_{m_j}\|}, \\
	\end{eqnarray*}
	where we used the result of Lemma~\ref{lemma:accmodelnextstep} \revised{(where 
	$\kappa_{efs}$ is defined)} and 
	Assumption~\ref{assum:modeldecrease}.
	Using the first part of~\eqref{eq:mudecreasegoodmodel}, we then have 
	$\mu_j\|J_{m_j}^\top r_{m_j}\|  \ge \kappa_{J_m}^2$, and therefore
	\begin{eqnarray*}
		f(x_{j+1}) - f(x_j) 
		&\le 
		& \frac{\kappa_{ef}+\kappa_{efs}}{\mu_j^2} -\frac{\eta_1 \theta_{fcd}}{2}\,
		\frac{\|J_{m_j}^\top r_{m_j}\|^2}{2\mu_j\|J_{m_j}^\top r_{m_j}\|} \\
		&= &\frac{1}{\mu_j}\left[\frac{\kappa_{ef}+\kappa_{efs}}{\mu_j} - 
		\frac{\eta_1 \theta_{fcd}}{4}\,\|J_{m_j}^\top r_{m_j}\|\right] 
		\le 
		\frac{1}{\mu_j}\left[ -\frac{\eta_1\theta_{fcd}}{8}\|J_{m_{j}}^\top r_{m_j}\|\right],
	\end{eqnarray*}
	where the second part of the maximum in~\eqref{eq:mudecreasegoodmodel} was used 
	in the last line, yielding~\eqref{eq:decreasegoodmodel}.
\end{proof}

The next result is a consequence of Lemma~\ref{lemma:decreasegoodmodel}.

\begin{lemma} \label{lemma:decreasegoodmodeltruegrad}
	Let the assumptions of Lemma~\ref{lemma:decreasegoodmodel} hold. 
	If $m_j$ is $(\kappa_{ef},\kappa_{eg})$-first-order accurate and
	\begin{equation} \label{eq:mudecreasegoodmodeltruegrad}
		\mu_j \ge \left( \kappa_{eg} + \max\left\{\kappa_{J_m}^2,
		\frac{8(\kappa_{ef}+\kappa_{efs})}{\eta_1 \theta_{fcd}}\right\}\right)
		\frac{1}{\|J(x_j)^\top r(x_j)\|},
	\end{equation}
	then the trial step $s_j$ satisfies
	\begin{equation} \label{eq:decreasegoodmodeltruegrad}
		\|r(x_j+s_j)\|^2 - \|r(x_j)\|^2 \; \le \; 
		-C_1 \frac{\|J(x_j)^\top r(x_j)\|}{\mu_j},
	\end{equation}
	where $C_1   := \frac{\eta_1\theta_{fcd}}{4}\frac{\max\left\{\kappa_{J_m}^2,
		\tfrac{4(\kappa_{ef}+\kappa_{efs})}{\eta_1 \theta_{fcd}}\right\}}
		{\kappa_{eg} + \max\left\{\kappa_{J_m}^2,
		\tfrac{4(\kappa_{ef}+\kappa_{efs})}{\eta_1 \theta_{fcd}}\right\}}$.
	Moreover, 
	\begin{equation} \label{eq:decreasegoodmodeltruegradscaled}
		\|r(x_j+s_j)\|^{1/2^\icc} - \|r(x_j)\|^{1/2^\icc} \; \le \; 
		-\frac{C_1}{2^{\icc+1}} \frac{\|J(x_j)^\top r(x_j)\|}{\|r(x_j)\|^{\ipow}\mu_j},
	\end{equation}	
\end{lemma}

\begin{proof}
	Since the model is $(\kappa_{ef},\kappa_{eg})$-first-order accurate, we have
	\begin{equation} \label{eq:nabfgkap}
		\left\|J(x_j)^\top r(x_j)\right\| \; \le \; 
		\left\|J(x_j)^\top r(x_j) - J_{m_j}^\top r_{m_j}\right\| + 
		\left\|J_{m_j}^\top r_{m_j}\right\| \; \le \; 
		\frac{\kappa_{eg}}{\mu_j} + 
		\left\|J_{m_j}^\top r_{m_j}\right\|.
	\end{equation}
	Using~\eqref{eq:mudecreasegoodmodeltruegrad} to bound the left-hand side, 
	we obtain:
	\begin{eqnarray*}
		\frac{\kappa_{eg} + \max\left\{\kappa_{J_m}^2,
		\frac{8(\kappa_{ef}+\kappa_{efs})}{\eta_1 \theta_{fcd}}\right\}}{\mu_j} 
		&\le &\frac{\kappa_{eg}}{\mu_j} + 
		\left\|J_{m_j}^\top r_{m_j}\right\| \\
	\end{eqnarray*}
	which gives $	\mu_j \ge  \max\left\{\kappa_{J_m}^2,
		\frac{8(\kappa_{ef}+\kappa_{efs})}{\eta_1 \theta_{fcd}}\right\}
		\frac{1}{\|J_{m_j}^\top r_{m_j}\|}$. We are thus in 
	the assumptions of Lemma~\ref{lemma:decreasegoodmodel}, 
	and~\eqref{eq:decreasegoodmodel} holds.
	Using the fact that the model is 
	$(\kappa_{ef},\kappa_{eg})$-first-order accurate together 
	with~\eqref{eq:mudecreasegoodmodeltruegrad} and~\eqref{eq:nabfgkap}, we have:
	\begin{equation*}
		\left\| J(x_j)^\top r(x_j) \right\| 
		\le \frac{\kappa_{eg}}{\mu_j} + 
		\left\| J_{m_j}^\top r_{m_j} \right\| 
		\le \frac{\kappa_{eg}}{\kappa_{eg} +  \max\left\{\kappa_{J_m}^2,
		\tfrac{8(\kappa_{ef}+\kappa_{efs})}{\eta_1 \theta_{fcd}}\right\}}
		\left\| J(x_j)^\top r(x_j) \right\|  
		+ \left\| J_{m_j}^\top r_{m_j} \right\| ,
	\end{equation*}
	leading to
	\begin{equation*}
		\|J_{m_j}^\top r_{m_j}\|
		\ge  \frac{\max\left\{\kappa_{J_m}^2,
		\tfrac{8(\kappa_{ef}+\kappa_{efs})}{\eta_1 \theta_{fcd}}\right\}}
		{\kappa_{eg} + \max\left\{\kappa_{J_m}^2,
		\tfrac{8(\kappa_{ef}+\kappa_{efs})}{\eta_1 \theta_{fcd}}\right\}}
		\|J(x_j)^\top r(x_j)\|.
	\end{equation*}
	Combining this relation with~\eqref{eq:decreasegoodmodel} finally 
	gives~\eqref{eq:decreasegoodmodeltruegrad}.
	To obtain~\eqref{eq:decreasegoodmodeltruegradscaled}, we simply invoke 
	\revised{Lemma~\ref{lemma:oursab}} starting from~\eqref{eq:decreasegoodmodeltruegrad}.
 \end{proof}

\begin{lemma}  \label{lemma:condmodelgraditsucc}
	Let Assumptions~\ref{assum:fC11}, \ref{assum:Jacmodel}, 
	\ref{assum:modeldecrease} and \ref{assum:stepbounds} hold. Consider the 
	$j$-th iteration of a realization of Algorithm~\ref{alg:LM}.
	\revised{Suppose that $m_j$ is $(\kappa_{ef},\kappa_{eg})$-first-order 
	accurate, $(r^0_j,r^s_j)$ is $\varepsilon_f$-accurate, and that}
	\begin{equation} \label{eq:condmodelgraditsucc}
		\mu_j \; \ge \; \max\left\{
		\frac{\alpha+\sqrt{\alpha^2+4\alpha\kappa_{J_m}^2(1-\eta_1)}}
		{2(1-\eta_1)} ,\eta_2 \right\} \frac{1}{\|J_{m_j}^\top r_{m_j}\|} :=
		\frac{\kappa_{\mu g}}{\|J_{m_j}^\top r_{m_j}\|},
	\end{equation}
	holds, where 
	$\alpha  := \varepsilon_f+\kappa_{eg}+\nu+ 3\kappa_{J_m}^2+\theta_{in}$.
	Then, the $j$-th iteration is successful (i.e.  $\rho_j \ge \eta_1$ and 
	$\|J_{m_j}^\top r_{m_j}\| \geq \frac{\eta_2}{\mu_j}$). 
\end{lemma}

\begin{proof}
To simplify the notations, we will omit the indices $j$ in the proof. By 
definition of the ratio $\rho$ and the model $m$, we have:
\begin{eqnarray*}
	\left|1-\frac{\rho}{2}\right| = \left|1-
	\frac{1}{2}\frac{f^0-f^s}{m(x)-m(x+s)\revised{-\tfrac{\gamma}{2}\|s\|^2}} \right|
	&= &\frac{\left| m(x)-m(x+s) -\revised{\tfrac{\gamma}{2}\|s\|^2} - \tfrac{1}{2}f^0 + \tfrac{1}{2}f^s\right|}
	{\left|m(x)-m(x+s)\revised{-\tfrac{\gamma}{2}\|s\|^2} \right|} \\
	&= &\frac{\left| -g_m^\top s - \tfrac{1}{2}s^\top(J_m^\top J_m + \gamma I)s 
	- \tfrac{1}{2}f^0 + \tfrac{1}{2}f^s \right|}{\left|m(x)-m(x+s)\revised{-\tfrac{\gamma}{2}\|s\|^2} \right|} \\
	&= &\frac{\left| \tfrac{1}{2}\left( f^s-f^0 - g_m^\top s - 
	s^\top J_m^\top J_m s \right) - \tfrac{1}{2}s^\top(g_m+\gamma s) \right|}
	{\left| m(x) - m(x+s)\revised{-\tfrac{\gamma}{2}\|s\|^2}  \right|} \\
	&\le &\frac{\tfrac{1}{2}\left|f^s-f^0 - g_m^\top s - 
	s^\top J_m^\top J_m s \right|+ \tfrac{1}{2}|s^\top(g_m+\gamma s)|}
	{\left| m(x) - m(x+s) \revised{-\tfrac{\gamma}{2}\|s\|^2} \right|}.
\end{eqnarray*}
The first term in the numerator can be bounded using a Taylor expansion of 
$f(x+s)$. Indeed,
\begin{eqnarray*}
	\left|f^s-f^0 - g_m^\top s - s^\top J_m^\top J_m s \right|
	&= &\left| f^s - f(x+s) + f(x) - f^0 + f(x+s) - f(x) - g_m^\top s 
	-\tfrac{1}{2}s^\top J_m^\top J_m s \right| \\
	&\le &\left| f^s - f(x+s)\right| + \left|f(x) - f^0 \right| 
	+\left| f(x+s) - f(x) - g_m^\top s 
	-\tfrac{1}{2}s^\top J_m^\top J_m s \right| \\
	&\le &\left| f^s - f(x+s) \right| + \left|f(x) - f^0 \right|  
	+ \left|\left[\nabla f(x) - g_m\right]^\top s \right| 
	+ \frac{\nu + \|J_m^\top J_m\|}{2}\|s\|^2.
\end{eqnarray*}
Recalling that $\nabla f(x)=J(x)^\top r(x)$, $g_m=J_m^\top r_m$ and 
using the accuracy properties of the model and function estimates as 
well as Assumption~\ref{assum:Jacmodel}, we obtain:
\begin{equation*}
	\left|f^s-f^0 - g_m^\top s - s^\top J_m^\top J_m s \right| \; \le \; 
	\frac{2\varepsilon_f}{\mu^2} + \frac{\kappa_{eg}\|s\|}{\mu}  +
	\frac{\nu+\kappa_{J_m}^2}{2}\|s\|^2.
\end{equation*}
Thus, we have
\begin{eqnarray*}
	\left|1-\frac{\rho}{2}\right| &\le 
	&\frac{\tfrac{\varepsilon_f}{\mu^2} + \tfrac{\kappa_{eg}\|s\|}{2\mu}  +
	\tfrac{\nu+\kappa_{J_m}^2}{4}\|s\|^2 + 
	\tfrac{1}{2}\left|s^\top(g_m+\gamma s)\right|}{m(x)-m(x+s) \revised{-\tfrac{\gamma}{2}\|s\|^2}}.
\end{eqnarray*}
Using Assumption~\ref{assum:stepbounds} on the numerator and 
Assumption~\ref{assum:modeldecrease} on the denominator, we arrive at
\begin{eqnarray*}
	\left|1-\frac{\rho}{2}\right| 
	\le \frac{\tfrac{\varepsilon_f}{\mu^2} + \tfrac{\kappa_{eg}}{\mu^2} 
	+ \tfrac{\nu+\kappa_{J_m}^2}{\mu^2} + 
	\tfrac{2 \kappa_{J_m}^2 + \theta_{in}}{\mu^2}}{m(x) - m(x+s)\revised{-\tfrac{\gamma}{2}\|s\|^2}} 
	&\le &\frac{\left(\varepsilon_f+\kappa_{eg}+\nu+ 3\kappa_{J_m}^2 
	+ \theta_{in}\right)\tfrac{1}{\mu^2}}{\tfrac{\theta_{fcd}}{2}
	\tfrac{\|J_m^\top r_m\|^2}{\|J_m\|^2+\gamma}} \\
	&= &\frac{\left(\varepsilon_f+\kappa_{eg}+\nu+ 3\kappa_{J_m}^2 
	+ \theta_{in}\right)\tfrac{\|J_m^\top r_m\|^2}{\gamma^2}}{\tfrac{\theta_{fcd}}{2}
	\tfrac{\|J_m^\top r_m\|^2}{\kappa_{J_m}^2+\gamma}} \\
	&= &\left(\varepsilon_f+\kappa_{eg}+\nu+ 3\kappa_{J_m}^2 
	+ \theta_{in}\right)\frac{\kappa_{J_m}^2+\gamma}{\gamma^2} 
	= \alpha\,\frac{\kappa_{J_m}^2+\gamma}{\gamma^2},
\end{eqnarray*}
\revised{where the last equality uses the definition of $\alpha$ from
the lemma's statement.}

As a result, we have
\begin{eqnarray*}
	\left| 1 - \frac{\rho}{2} \right| \ge 1- \eta_1 &\Rightarrow 
	&\alpha\frac{\kappa_{J_m}^2+\gamma}{\gamma^2} \ge 
	1-\eta_1 \\
	&\Leftrightarrow &0 \ge (1-\eta_1)\gamma^2 - 
	\alpha\gamma - \alpha\,\kappa_{J_m}^2.
\end{eqnarray*}

Since the right-hand side is a second-order polynomial in $\gamma$, this gives
	\[
		\gamma \le \frac{\alpha+\sqrt{\alpha^2+4\alpha\kappa_{J_m}^2(1-\eta_1)}}
		{2(1-\eta_1)} \; \Leftrightarrow \;
		\mu \le \frac{\alpha+\sqrt{\alpha^2+4\alpha\kappa_{J_m}^2(1-\eta_1)}}
		{2(1-\eta_1)}\frac{1}{\|J_m^\top r_m\|}.
	\]
But this contradicts~\eqref{eq:condmodelgraditsucc}, from which we 
conclude that we necessarily have 
$\left| 1 - \frac{\rho}{2}\right| < 1- \eta_1$, and thus $\rho > \eta_1$. 
Since $\|J_m^\top r_m\| \ge \frac{\eta_2}{\mu}$ as a direct consequence 
of~\eqref{eq:condmodelgraditsucc}, the iteration is a successful one, and the 
parameter $\mu$ is not increased.
\end{proof}

The next lemma shows that having accurate function estimates but 
inaccurate models still leads to a decrease in the residual on successful iterations.

\begin{lemma} \label{lemma:decreasegoodestim}
	Let Assumptions~\ref{assum:fC11}, \ref{assum:Jacmodel}, 
	\ref{assum:modeldecrease}, and \ref{assum:stepbounds} hold. 
	\revised{Given a realization of Algorithm~\ref{alg:LM}, let $j$ be the index 
	of an unsuccessful iteration.}
%
	Suppose that $(r^0_j,r^s_j)$ is $\varepsilon_f$-accurate, \revised{and suppose}
	\begin{equation} \label{eq:eta2bound}
		\eta_2 \ge \max\left\{\kappa_{J_m}^2,\frac{8\varepsilon_f}{\eta_1\theta_{fcd}}\right\}
	\end{equation}
	Then, one has:
	\begin{equation} \label{eq:decreasegoodestim}
		\|r(x_j+s_j)\|^2 - \|r(x_j)\|^2 \; \le \; -\frac{C_2}{\mu_j^2}.
	\end{equation}
	where $C_2  :=\tfrac{\eta_1\eta_2\theta_{fcd}}{2}-4\varepsilon_f > 0$.
	Moreover, \revised{if $\|r(x_j)\| > 0$, we have}
	\begin{equation} \label{eq:decreasegoodestimscaled}
		\|r(x_j+s_j)\|^{1/2^\icc} - \|r(x_j)\|^{1/2^\icc} \; \le \; 
		-\frac{C_2}{2^{\icc+1}} \frac{1}{\mu_j^2 \|r(x_j)\|^{\ipow}}.
	\end{equation}
\end{lemma}

\begin{proof}
	By definition of a successful iteration and using the accuracy properties 
	of the models and the estimates, we have
	\begin{eqnarray*}
		f(x_{j+1}) - f(x_j) = f(x_j+s_j) - f(x_j) &= 
		&\frac{1}{2}\left[\|r(x_j+s_j)\|^2-\|r^s_j\|^2 
		+ \|r^s_j\|^2 -\|r^0_j\|^2 + \|r^0_j\|^2 - \|r(x_j)\|^2 \right] \\
		&\le &2\frac{\varepsilon_f}{\mu_j^2} + \frac{1}{2}\|r^s_j\|^2 - \frac{1}{2}\|r^0_j\|^2 \\
		&\le &2\frac{\varepsilon_f}{\mu_j^2} + \eta_1\left(m(x_j+s_j)-m(x_j)\right) \\
		&\le &2\frac{\varepsilon_f}{\mu_j^2} -\frac{\eta_1 \theta_{fcd}}{2}\,
		\frac{\|J_{m_j}^\top r_{m_j}\|^2}{\kappa_{J_m}^2+\mu_j\|J_{m_j}^\top r_{m_j}\|} \\
		&\le &2\frac{\varepsilon_f}{\mu_j^2} -\frac{\eta_1 \theta_{fcd}}{2}\,
		\frac{\|J_{m_j}^\top r_{m_j}\|^2}{\eta_2+\mu_j\|J_{m_j}^\top r_{m_j}\|} 
		\qquad \mbox{since $\eta_2 \ge \kappa_{J_m}^2$.}
	\end{eqnarray*}
	Since the iteration is successful, we have $\mu_j\|J_{m_j}^\top r_{m_j}\| \ge \eta_2$, leading 
	to
	\begin{equation*}
	\|r(x_{j+1})\|^2 - \|r(x_j)\|^2 \le 4\frac{\varepsilon_f}{\mu_j^2} 
	-\frac{\eta_1 \theta_{fcd}}{2}\,\frac{\|J_{m_j}^\top r_{m_j}\|}{\mu_j} 
	\le 4\frac{\varepsilon_f}{\mu_j^2} -\frac{\eta_1\eta_2 \theta_{fcd}}{2}\,
	\frac{1}{\mu_j^2} 
	= -\frac{C_2}{\mu_j^2}.
	\end{equation*}
	(Note that the positivity of $C_2$ comes from~\eqref{eq:eta2bound}). 
	This establishes the desired result~\eqref{eq:decreasegoodestim}; by applying 
	\revised{Lemma~\ref{lemma:oursab}}, we then arrive at~\eqref{eq:decreasegoodestimscaled}.
 \end{proof}

To end this section, we formalize our assumptions regarding the probabilistic 
properties satisfied by our method.

\begin{assumption} \label{assum:models}
	The random model sequence $\{M_j\}$ is $p$-probabilistically 
	$\{\kappa_{ef},\kappa_{eg}\}$-first-order accurate for some $p \in (0,1]$, 
	$\kappa_{ef} > 0$, and $\kappa_{eg} > 0$.
\end{assumption}

\begin{assumption} \label{assum:festimates}
	The sequence of random function estimates $\{(R^0_j,R^s_j)\}_j$ is 
	$q$-probabilistically $\varepsilon_f$-accurate for some $q \in (0,1]$ 
	and $\varepsilon_f > 0$.
\end{assumption}
\begin{assumption} \label{assum:eta2}
	The constant $\eta_2$ is chosen such as  
	\begin{equation} \label{eq:condeta2cvmuseries}
		\eta_2 \; \ge \; \max\left\{ \kappa_{J_m}^2, 
		\frac{6(\kappa_{ef}+\kappa_{efs})}{\theta_{fcd}},
		\frac{8\varepsilon_f}{\eta_1\theta_{fcd}}\right\}.
	\end{equation}
\end{assumption}
In the rest of the paper, we will assume that $pq \neq 1$, since if $pq=1$, then for 
for every $j$, we have $p_j^* \; = \; P(\Uj_j | \mathcal{F}^M_{j-1}) \; = \; p   \; = \; 
	q_j^* \; = \; P(\Vj_j | \mathcal{F}^{M \cdot R}_{j-1}) \;  \; = \; \; q  \; = \; 1. $
and the behavior of the algorithm reduces to that of a deterministic 
method. Note that we still allow $p$ or $q$ to be equal to 1.

\subsection{A key property} \label{subsec:keythm}

Similarly to existing analyzes, the main task in deriving our complexity 
result consists in proving the following theorem.

\begin{theorem} \label{theo:cvmuseries}
	Let Assumptions~\ref{assum:fC11}, \ref{assum:Jacmodel}, \ref{assum:modeldecrease}, \ref{assum:stepbounds}
	and \ref{assum:eta2} hold. Suppose that Assumptions~\ref{assum:models} and 
	\ref{assum:festimates} are also satisfied, with the probabilities $p$ and $q$ 
	chosen in a way specified later on. 
	Then,
	\begin{equation} \label{eq:cvmuseries}
		\Pr\left( \sum_{j=0}^{\infty}\,
		\frac{1}{\Mu_j^2\,\|r(X_j)\|^{\ipow}} < \infty 
		\ \middle| \|r(X_j)\|^{\ipow} > 0\ \forall j\ 
		\right) \; = \; 1.
	\end{equation}
\end{theorem}
%
%
%
%
	Our proof technique (available in Appendix) is adapted from that in 
	the trust-region setting~\cite[Theorem 4.11]{RChen_MMenickelly_KScheinberg_2018}
	\cite[Lemma 4]{JLarson_SCBillups_2016}. It relies on a Lyapunov 
	function combining the parameter $\Mu_j$ and a measure of stationarity. 
	Previous analyzes, including an earlier\revised{, unpublished} version of this 
	paper~\cite{EBergou_YDiouane_VKungurtsev_CWRoyer_2018}, considered the 
	sequence
	\begin{equation} \label{eq:oldPhi_j}
		\tau f(X_j) + \frac{1-\tau}{\Mu_j^2}
	\end{equation}
	for an appropriately chosen $\tau \in (0,1)$.
	In order to employ our 
	scaled stationarity criterion, we fix an integer 
	$\icc \in \mathbb{N} \cup \{-1\}$ and define
	\begin{eqnarray}  \label{eq:Phi_j}
		\Phi_j &  := &  \tau \|r(X_j)\|^{1/2^\icc}
		+ \frac{1-\tau}{\Mu_j^2 \|r(X_j)\|^{\ipow}}.
	\end{eqnarray}
	When $\icc=-1$, we \revised{recover the} choice~\eqref{eq:oldPhi_j} up to 
	a constant factor. As we will see, the more generic 
	definition~\eqref{eq:Phi_j} is well suited to our use of the scaled 
	gradient~\eqref{eq:scaledgrad}.

	The proof requires $\tau \in (0,1)$ to be chosen such that
	\begin{equation} \label{eq:taumuseries}
		\frac{\tau}{1-\tau} \; > \;	\max\left\{ 
		\frac{2^{\icc+1}\left(\lambda^{2}-\tfrac{1}{\lambda^2}\right)}{C_1 \zeta}, 
		\frac{2^{\icc+1}\left(\lambda^{2}-\tfrac{1}{\lambda^2}\right)}{C_2},
		\frac{2^{\icc+2}\left(\lambda^{2}-\tfrac{1}{\lambda^2}\right)}{\kappa_{ef}+\kappa_{efs}}\right\}
	\end{equation}		
	where $\zeta$ is a parameter such that
	\begin{equation} \label{eq:zetamuseries}
		\zeta \; \ge \; \left(\kappa_{eg} + \max\left\{\kappa_{\mu g},
		\frac{8(\kappa_{ef}+\kappa_{efs})}{\eta_1 \theta_{fcd}},
		\kappa_{J_m}^2,\eta_2 \right\}\right).
	\end{equation}	
	In addition, the probabilities $p$ and $q$ are required to satisfy:
	\begin{equation}  \label{eq:condprobaspq1}
		\frac{pq-1/2}{(1-p)(1-q)} \ge \frac{C_3}{C_1}, 
		\revised{
		\quad \mbox{where} \quad C_3:= 4\left(1+\tfrac{\nu}{\zeta}\right).
		}
	\end{equation}	
	\revised{The bound~\eqref{eq:condprobaspq1}} implies that $pq \ge 1/2$, as well as
	\begin{equation} \label{eq:condprobaspq2}
		(1-p)(1-q) \le \frac{(1-\tau)\left(1-\tfrac{1}{\lambda^2}\right)}
		{2\left(\tau C_3\zeta + (1-\tau)(\lambda^{2}-1)\right)}.
	\end{equation}
	We remark that conditions~\eqref{eq:condprobaspq1} and~\eqref{eq:condprobaspq2} 
	will be satisfied for $p$ and $q$ sufficiently close to $1$; we also point 
	out that when $q=1$, these conditions would essentially reduce to 
	$p \ge 1/2$~\cite{EBergou_SGratton_LNVicente_2016}.
	
	Provided conditions~\eqref{eq:taumuseries}--\eqref{eq:condprobaspq2} hold, 
	Theorem~\ref{theo:cvmuseries} is obtained by proving that there exists 
	$\sigma > 0$ such that, at every iteration $j$,
	\begin{equation} \label{eq:expecteddecPhij}
		\E{\Phi_{j+1}-\Phi_j | \mathcal{F}^{M \cdot R}_{j-1} \cap E_j^0} 
		\; \le \; -\frac{\sigma}{\Mu_j^2\,\|r(X_j)\|^{\ipow}},
	\end{equation}
	where $\mathcal{F}^{M \cdot R}_{j-1} \cap E_j^0$ is the trace $\sigma$-algebra 
	\footnote{\revised{Given a $\sigma$-algebra $\mathcal{F}$ and an event $E$, the 
	trace $\sigma$-algebra is defined as $\{F \cap E | F \in \mathcal{F}\}$.}}
	produced by the event $E_j^0 =\{ \|r(X_k)\|^{\ipow} > 0\ \forall k=0,\dots,j\}$, 
	and the expectation is taken over the product trace $\sigma$-algebra generated by 
	all models and function value estimates. 
	We point out that the right-hand side is measurable with respect to 
	$\mathcal{F}^{M \cdot R}_{j-1}$.
	Since in our case, both $\|r(x_j)\|$ and $f(x_j)$ are bounded from below by $0$, we have that 
	$\Phi_j \ge 0$ and $\Mu_j>0$, 
	\eqref{eq:expecteddecPhij} guarantees that the series converges almost surely 
	(see, e.g.,~\cite[Proposition 4.24]{ECinlar_2011}).

	The proof focuses on a realization of the process $\Phi_k$, and divides
	the iterations into two subsets, depending on whether the following condition holds:
	\begin{equation} \label{eq:condproofmuseries}
		\|J(x_j)^\top r(x_j)\| \ge \frac{\zeta}{\mu_j}.
	\end{equation}
	This condition is strongly related to the requirements on $\mu_j$ in the lemmas 
	of Section~\ref{subsec:assumdeterm}.

%

\if\proofsCV1
\begin{proposition} \label{propo:submartingale}
Let $G_j$ be a submartingale, in other words, a set of random variables
which are integrable ($E(|G_j|) < \infty$) and satisfy
$E(G_j|\mathcal{F}_{j-1}) \geq G_{j-1}$, for every $j$,
where $\mathcal{F}_{j-1} := \sigma(G_0,\ldots,G_{j-1})$  is the $\sigma$-algebra generated by $G_0,\ldots,G_{j-1}$
and $E(G_j|\mathcal{F}_{j-1})$ denotes the conditional expectation of $G_j$ given the past history of
events $\mathcal{F}_{j-1}$.

Assume further that there exists $M>0$ such that $| G_j - G_{j-1} | \leq  M <\infty$, for every $j$.
Consider the random events $C := \{\lim_{j \to \infty} G_j$ ${\rm  exists\ and\ is\ finite}\}$
and $D := \{\limsup_{j\to \infty} {G_j} = \infty\}$. Then $P (C \cup D) = 1$.
\end{proposition}

This finally leads to the desired result.

\begin{theorem} \label{theo:cvliminf}
	Let the assumptions of Theorem~\ref{theo:cvmuseries} and Corollary~\ref{coro:condprobaspq} 
	hold. Then, the sequence of random iterates generated by Algorithm~\ref{alg:LM} satisfies:
	\[
		\Pr\left( \liminf_{j \rightarrow \infty} \|J(X_j)^\top r(X_j)\| = 0 
		\right) = 1.
	\]
\end{theorem}

\begin{proof}
	Following the lines of the proof of~\cite[Theorem 4.16]{RChen_MMenickelly_KScheinberg_2018}, we 
	proceed by contradiction and assume that, there exists $\epsilon' >0$ 
	such that 
	\[
		\Pr\left(\|\nabla f(X_j)\| \ge \epsilon' \forall j \; \middle| \; 
		\left\{ \Mu_j \rightarrow \infty \right\}\right) > 0.
	\]
	We then consider a realization of Algorithm~\ref{alg:LM} for which 
	$\|\nabla f(x_j)\| \ge \epsilon'\ \forall j$. Since 
	$\lim_{j \rightarrow \infty} \mu_j = \infty$, there exists $j_0$ such 
	that for every $j \ge j_0$, we have:
	\begin{equation} \label{eq:musuffbig}
		\mu_j  > b  :=\max\left\{ \frac{2\kappa_{\mu g}}{\epsilon'}, 
		\frac{16(\kappa_{ef}+\kappa_{efs})}{\eta_1 \theta_{fcd}\epsilon'}, 
		\frac{2\kappa_{J_m}^2}{\epsilon'}, \frac{2\eta_2}{\epsilon'}, 
		\lambda\,\mu_{\min}\right\}.
	\end{equation}
	
	Let \revised{$Z_j$} be a random variable with realizations 
	$\revised{z_j}=\log_{\lambda}\left(\frac{b}{\mu_j}\right)$: 
	then for the realization we are considering, we have $\revised{z_j} < 0$ for $j \ge j_0$. Our objective 
	is to show that such a realization has a a zero probability of occurrence.
	
	Consider $j \ge j_0$ such that both events $S_j$ and $\Vj_j$ happen: the probability of such 
	an event is at least $pq$. Because the model is accurate and we have~\eqref{eq:musuffbig}: 
	\[
		\|g_{m_j}\| \ge \|\nabla f(x_j)\| -\frac{\kappa_{eg}}{\mu_j^2} \ge 
		\epsilon' - \frac{\epsilon'}{2} = \frac{\epsilon'}{2}.
	\]	
	We are thus in the assumptions of Lemmas~\ref{lemma:decreasegoodmodeltruegrad} 
	and~\ref{lemma:condmodelgraditsucc}, from which we conclude 
	that the $j$-th iteration is successful, so the parameter $\mu_j$ is decreased, i.e., 
	$\mu_{j+1} = \tfrac{\mu_j}{\lambda}$. Consequently, \revised{$z_{j+1}=z_j+1$}. 
	
	For any other outcome for $\Uj_j$ and $\Vj_j$ other than ``both happen" (which occur with 
	probability at most $1-pq$), we have $\mu_{j+1} \le \lambda \mu_j$. As a result,
	letting $\mathcal{F}^{V\cdot T}_{j-1} = \sigma(V_0,\dots,V_{k-1}) \cap 
	\sigma(T_0,\dots,T_{k-1}) = \sigma(\revised{Z_0,\dots,Z_{j-1}})$,
	\[
		\E{\revised{Z_{j+1}} | \mathcal{F}^{V\cdot T}_{j-1}} \ge 
		pq (\revised{Z_j+1}) + (1-pq)(\revised{Z_j}-1) \ge \revised{Z_j},
	\]
	because $pq > 1/2$ as a consequence of the assumptions from 
	Corollary~\ref{coro:condprobaspq}. This implies that \revised{$Z_j$} is a submartingale.
	
	We now define another submartingale $W_j$ by
	\[
		W_j = \sum_{i=0}^j (2 \mathbf{1}_{\Uj_i} \mathbf{1}_{\Vj_i} - 1),
	\]
	where $\mathbf{1}_A$ is the indicator random variable of the event $A$.
	Note that $W_j$ is defined on the same probability space as \revised{$Z_j$}, and 
	that we have:
	\begin{eqnarray*}
		\E{W_j|\mathcal{F}^{V\cdot T}_{j-1}} &= 
		&\E{W_{j-1}|\mathcal{F}^{V\cdot T}_{j-1}} + 
		\E{2 \mathbf{1}_{\Uj_j} \mathbf{1}_{\Vj_j} - 1|\mathcal{F}^{V\cdot T}_{j-1}} \\
		&= &W_{j-1} + 2\Pr\left(\Uj_j \cap \Vj_j | \mathcal{F}^{V\cdot T}_{j-1}\right) - 1 \\
		&\ge &W_{j-1},
	\end{eqnarray*}
	where the last inequality holds because $pq  \ge 1/2$. Therefore, 
	$W_j$ is a submartingale with bounded ($\pm 1$) increments. By 
	Proposition~\ref{propo:submartingale}, it does not have a finite limit and 
	the event $\left\{ \limsup_{j \rightarrow \infty} W_j = \infty\right\}$ has 
	probability 1.
	
	To conclude, observe that by construction of \revised{$Z_j$} and $W_j$, one has 
	$\revised{z_j-z_{j_0}} \ge w_j  - w_{j_0}$, where $w_j$ is a realization of $W_j$. 
	This means that $R_j$ must be positive infinitely often with probability 
	one, thus that there is a zero probability of having $r_j < 0 ~~~\forall 
	j \ge j_0$. This contradicts our initial assumption that 
	$\Pr(\|\nabla f(X_j)\| \ge \epsilon' ~~~\forall j ) > 0$, which means 
	that we must have 
	\[
		\Pr\left( \liminf_{j \rightarrow \infty} \|\nabla f(X_j)\| = 0 
		\right) = 1.
	\]
 \end{proof}

When $\{J(X_j)\}$ is non-degenerate, we note that Theorem~\ref{theo:cvliminf} 
implies almost-sure convergence of the residual norm sequence $\{\|r(X_j)\|\}$.
\fi
%
Note that the result of Theorem~\ref{theo:cvmuseries} holds for any 
$\icc \in \mathbb{N} \cup {-1}$. In particular, when $\icc=-1$, the 
conditioning event $\|r(X_j)\|^{\ipow}=1 >0$ is true for all realizations of 
the method and we simply have:
\[
	\Pr\left( \sum_{j=0}^{\infty}\,\frac{1}{\Mu_j^2} < \infty \right) \; = \; 1.
\]
\revised{As a result, one can follow the lines of the proof 
of~\cite[Theorem 4.16]{RChen_MMenickelly_KScheinberg_2018} and establish that
Algorithm~\ref{alg:LM} to a stationary point, i.e. the sequence of iterates 
satisfies $\liminf_j \|J(X_j)^\top r(X_j)\| = 0$,  with 
probability 1. In this work, we focus is on complexity results, thus we refer 
to \revised{the first, unpublished} version of 
this work for a full convergence 
proof~\cite{EBergou_YDiouane_VKungurtsev_CWRoyer_2018}.}

\subsection{Complexity bound} \label{subsec:ww:mainres}

We now introduce the necessary probability tools to derive our complexity 
results. Given a stochastic process $\{X_j \}$, $T$ is said to a be a 
\emph{stopping time} for $\{X_j \}$, if, for all $j \ge 1$, the event 
$\{ T \le j\}$ belongs to the $\sigma$-algebra associated with $X_1, X_2, ...X_j$. For a 
given $\epsilon>0$, we define a random time $T_\epsilon$ by
$$
T_\epsilon := \inf \left\{j \ge 0, \|r(X_j)\| \le \epsilon_p\ \mbox{or}\ 
\|g_r^{\icc}(X_j)\| \le \epsilon_d \right\}.
$$
We also define $\Mu_\epsilon = \zeta/(\epsilon_p^{\ipow}\epsilon_d)$, 
where $\zeta$ satisfies~\eqref{eq:zetamuseries}.
Based on the above analysis, $T_{\epsilon}$ is a stopping time for the 
stochastic process defined by Algorithm \ref{alg:LM} and hence for 
$\{\Phi_j, \Mu_j\}$ where $\Phi_j$ is given by (\ref{eq:Phi_j}).

\begin{assumption} \label{assum:bnds:mu&Phi}
	There exists a positive constant $\kappa_{r}>0$ such for all $j$, 
	w.p. 1, $\|r(X_j)\| \le \kappa_r$.
\end{assumption}


By Assumption~\ref{assum:bnds:mu&Phi}, we have that for any $j$ such that 
$\|r(X_j)\| \ge \epsilon_p$:
\begin{equation*}
	\forall j,\quad \Phi_j = \tau \|r(X_j)\|^{1/2^\icc} + 
	(1-\tau) \frac{1}{\Mu_j^2 \|r(X_j)\|^{\ipow}} \;
	\le \kappa_r^{1/2^\icc} + \frac{1}{\mu_{\min}^2\epsilon_p^{\ipow}} \;
	\le \Phi_{\max}\epsilon_p^{\frac{1-2^{\icc+1}}{2^\icc}},
\end{equation*}
where we define $\Phi_{max}=\kappa_r^{1/2^\icc}+\tfrac{1}{\mu_{\min}^2}$. The 
last inequality comes from $\ipow \ge 0$ for $\icc \ge -1$ together with
$\epsilon_p \in (0,1)$.
}

For simplicity reasons, we will assume that 
$\mu_0= \frac{\Mu_\epsilon}{\lambda^s}$ and 
$\mu_{\min}= \frac{\Mu_\epsilon}{\lambda^t}$ for some integers $s,t>0$, hence 
for all $j$, one has $\Mu_j= \frac{\Mu_\epsilon}{\lambda^k}$ for some integer 
$k$. We note that, in this case, whenever  $\Mu_{j} < \Mu_\epsilon$, one has 
$\Mu_{j} \le \frac{\Mu_\epsilon}{\lambda}$, and hence 
$\Mu_{j+1} \le\Mu_\epsilon$. This assumption can be made without loss of 
generality, for instance, provided $\mu_{\min} = \mu_0\lambda^{s-t}$ (one can 
choose $\mu_{\min}$ so that this is true) and $\zeta = \mu_0\lambda^s\epsilon$, 
where $s$ is the smallest integer such that $\zeta$ 
satisfies~\eqref{eq:zetamuseries}.

The next lemma defines a geometric random walk based on successful iterations. 
The final complexity result heavily depends upon the behavior of this random 
walk. Note that this reasoning departs from the existing analysis of 
stochastic trust-region 
schemes~\cite{JBlanchet_CCartis_MMenickelly_KScheinberg_2019}.

\begin{lemma} \label{lm:stoch:Mu_update} 	
Let Assumptions~\ref{assum:fC11}, \ref{assum:Jacmodel}, 
\ref{assum:modeldecrease}, \ref{assum:stepbounds} and \ref{assum:eta2} hold. 
For all $j < T_{\epsilon}$, whenever $\Mu_{j} \ge \Mu_\epsilon$, one has
$\Mu_{j+1}= \frac{\Mu_j}{\lambda} \mathbf{1}_{\Omega_j} 
+ \lambda \Mu_{j} (1 -\mathbf{1}_{\Omega_j}),
$ or, equivalently, letting $\gamma = \log(\lambda)$, one has
\begin{eqnarray} \label{stoch:Mu_update}
\Mu_{j+1} & = &  \Mu_j e^{\gamma  \Lambda_j},
\end{eqnarray}
where $\mathbf{1}_{\Omega_j}$ is equal to $1$ if the iteration $j$ is 
successful and $0$ otherwise, and $\Lambda_j = 2 \mathbf{1}_{\Omega_j}-1$ 
defines a birth-and-death process\revised{. That is, $\{\Lambda_j\}$ satisfies}
$$
\Pr(\Lambda_j= 1 |\mathcal{F}_{j-1}^{M \cdot R}, \Mu_{j} \ge \Mu_\epsilon) 
= 1 - \Pr(\Lambda_j=-1 |\mathcal{F}_{j-1}^{M \cdot R}, \Mu_{j} \ge \Mu_\epsilon) 
=  \omega_j, \quad \mbox{with}\quad \omega_j \ge pq.
$$
\end{lemma}
\begin{proof}
By the mechanism of the algorithm one has 
$\Mu_{j+1} =  \frac{\Mu_j}{\lambda} \mathbf{1}_{\Omega_j} 
+ \lambda \Mu_{j} (1 -\mathbf{1}_{\Omega_j})$.

Moreover, if $\mu_{j} \ge \Mu_\epsilon$ for a given $j< T_{\epsilon}$, one has 
\[
	\| \nabla f(x_j)\| = \|g_r(x_j)\|\|r(x_j)\|^{\ipow} \ge \epsilon_d \epsilon_p^{\ipow} 
	=\frac{\zeta}{\Mu_\epsilon} \ge \frac{\zeta}{\mu_j}.
\]
Assume $\mathbf{1}_{\Uj_j}=1$ and $\mathbf{1}_{\Vj_j}=1$ (i.e. both $m_j$ and 
$(r^0_j,r^s_j)$ are accurate). Because the model is 
$(\kappa_{ef},\kappa_{eg})$-first-order accurate, this implies
\[
	\|g_{m_j}\| \ge \|\nabla f(x_j)\|-\frac{\kappa_{eg}}{\mu_j} 
	\ge (\zeta-\kappa_{eg})\frac{1}{\mu_j} \ge \frac{\kappa_{\mu g}}{\mu_j},
\] 
and since the estimates are also accurate, the iteration is must be successful per 
Lemma~\ref{lemma:condmodelgraditsucc}. Hence, one gets 
$\omega_j= \Pr(\Lambda_j=1 |\mathcal{F}_{j-1}^{M \cdot R}, \Mu_{j} \ge \Mu_\epsilon) 
\ge pq$.
\end{proof}
Lemma~\ref{lm:stoch:Mu_update} is analogous to~\cite[Lemma 7]{JBlanchet_CCartis_MMenickelly_KScheinberg_2019}, 
however, in our case, the birth-and-process $\{\Lambda_j\}$ is based on successful iterations, whereas~\cite{JBlanchet_CCartis_MMenickelly_KScheinberg_2019} considered the iterations where both the function estimates 
and the model were accurate.

For convenience, conditioned on the fact that
$T_{\epsilon}>j$, the following proposition recalls the main argument in
proving Theorem~\ref{theo:cvmuseries} in the case where $\|J(X_j)^\top r(X_j)\| \ge \tfrac{\zeta}{\Mu_j}$ (see ``\textit{Case 1}'' in the proof of Theorem \ref{theo:cvmuseries} in Appendix~\ref{sec:prooftheomuseries}).

\begin{proposition} 
\label{lm:decrease:phi} Let Assumptions~\ref{assum:fC11}, \ref{assum:Jacmodel}, 
\ref{assum:modeldecrease}, \ref{assum:stepbounds} and \ref{assum:eta2} hold.
Suppose that Assumptions~\ref{assum:models} and \ref{assum:festimates} are 
also satisfied\revised{. Moreover, assume that} the probabilities $p$ and $q$ satisfy:
	\begin{equation}  \label{eq:condprobaspq11}
		\frac{pq-1/2}{(1-p)(1-q)} \ge \frac{C_3}{C_1},
	\end{equation}	
and
	\begin{equation} \label{eq:condprobaspq22}
		(1-p)(1-q) \le \frac{(1-\tau)\left(1-\tfrac{1}{\lambda^2}\right)}
		{2\left(\tau C_3\zeta + (1-\tau)(\lambda^{2}-1)\right)}.
	\end{equation} 
Then, there exists a constant $\sigma>0$ such that, conditioned on 
$T_{\epsilon}>j$, one has
	\begin{eqnarray}
		\E{\Phi_{j+1}-\Phi_j | \mathcal{F}_{j-1}^{M \cdot R}} 
		= \E{\Phi_{j+1} | \mathcal{F}_{j-1}^{M \cdot R}} - \Phi_j
		< -\frac{\sigma}{\Mu_j^2 \|r(X_j)\|^{\ipow}},
	\end{eqnarray}
%
where $\sigma := \tfrac{1}{4}(1-\tau)\left(1-\tfrac{1}{\lambda^2}\right)$.
\end{proposition}

We now define a renewal process $\{A_i\}$ by $A_0=0$ and 
$A_i=\min \left\{k> A_{i-1}: \Mu_k \le \Mu_\epsilon\right\}$, that
\revised{tracks the iterations indices for which $\Mu_j$ returns to 
a prescribed level $\Mu_\epsilon$.}
For all $j\ge 1$, we let
$\tau_j = A_j - A_{j-1}.$ The next result provides a bound on the expected
value of $\tau_j$.

\begin{lemma} \label{lm:E_tau}
Let Assumptions~\ref{assum:fC11}, \ref{assum:Jacmodel}, 
\ref{assum:modeldecrease}, \ref{assum:stepbounds} and \ref{assum:eta2} hold. 
Assuming that $pq> \frac{1}{2}$, one has for all $j$
\begin{eqnarray} \label{lem:tau}
\E{\tau_j} &\le & \frac{pq}{2pq-1}.
 \end{eqnarray}
\end{lemma}
 \begin{proof}
 One has 
\begin{eqnarray} \label{eq:1}
 \E{\tau_j} &=& \E{\tau_j | \Mu_{A_{j-1}} < \Mu_\epsilon} \Pr(\Mu_{A_{j-1}} < \Mu_\epsilon) +  \E{\tau_j | \Mu_{A_{j-1}} = \Mu_\epsilon} \Pr(\Mu_{A_{j-1}} = \Mu_\epsilon) \nonumber \\
 & \le & \max\{ \E{\tau_j | \Mu_{A_{j-1}} < \Mu_\epsilon} , \E{\tau_j | \Mu_{A_{j-1}} = \Mu_\epsilon} \}.
 \end{eqnarray}
 First we note that whenever  $\Mu_{j} < \Mu_\epsilon$, one has $\Mu_{j} \le \frac{\Mu_\epsilon}{\lambda}$, and hence $\Mu_{j+1} \le\Mu_\epsilon$. Thus, if $\Mu_{A_{j-1}} < \Mu_\epsilon $, one deduces that $A_j=A_{j-1} +1$ and then 
 \begin{eqnarray} \label{eq:2}
 \E{\tau_j | \Mu_{A_{j-1}} < \Mu_\epsilon}&=&1.
\end{eqnarray}
Assuming now that $A_{j}>A_{j-1}+1$ (if not, meaning that $A_{j}=A_{j-1}+1$, the proof is straightforward ), then conditioned on $\Mu_{A_{j-1}} = \Mu_\epsilon$, one has $ \Mu_{A_{j}} = \Mu_\epsilon$ as well. We note also that for all $k_j \in [A_{j-1},A_{j}]$, one has  $\Mu_{k_j} \ge \Mu_\epsilon$. Hence, using Lemma \ref{lm:stoch:Mu_update}, one has
\begin{eqnarray*} 
\Mu_{k_j+1} & = &  \Mu_{k_j} e^{\gamma  \Lambda_{k_j}},
\end{eqnarray*}
where $\gamma =\log(\lambda)$ and $\Pr(\Lambda_{k_j}=1 |\mathcal{F}_{{k_j}-1}^{M \cdot R}, \Mu_{{k_j}} \ge \Mu_\epsilon) = \omega_{k_j}$ and $\Pr(\Lambda_{k_j}=-1 |\mathcal{F}_{{k_j}-1}^{M \cdot R}, \Mu_{{k_j}} \ge \Mu_\epsilon) = 1- \omega_{k_j}$. Moreover, one has $\omega_{k_j}\ge pq$.

The process 
$\{\Mu_{A_{j-1}},\Mu_{A_{j-1}+1}, \ldots, \Mu_{A_{j}} \}$ then defines a geometric random walk between two returns to the same state (i.e., $ \Mu_\epsilon$) and $\tau_j$ represents the number of iterations until a return to the initial state.
%
%
For such a geometric random walk, one can define the state probability vector $\pi=(\pi_k)_k$ corresponding to the limiting stationary 
distribution~\cite{RDYates_DJGoodman_2005}.  Using the local balance equation between the two states $k$ and $k+1$, see \cite[Theorem 12.13]{RDYates_DJGoodman_2005}, one has
$$
(1-\omega_k)\pi_k = \omega_k \pi_{k+1}.
$$
Since $\omega_k \ge pq$, one deduces that $(1-pq)\pi_k \ge pq \pi_{k+1}$.
Hence, 
$$
\pi_k \le  \varpi^k \pi_0  \mbox{~~~~where~~} \varpi=\frac{1-pq}{pq}.
$$
Using the assumption $ \varpi<1$ (i.e. $ pq>\frac{1}{2}$) and the definition of the state probability $\sum_k^\infty \pi_k =1$, one has $\pi_0 \ge 1- \varpi$ (this is a classical result for geometric random walk, see for instance  \cite[Example 12.26]{RDYates_DJGoodman_2005}).

Applying the properties of ergodic Markov chains, one deduces that the expected number of iterations until a return to the initial state (the state $0$) is given by $\frac{1}{\pi_0}$.
Hence 
\begin{equation} \label{eq:3}
\E{\tau_j | \Mu_{A_{j-1}} = \Mu_\epsilon} =  \frac{1}{\pi_0} \le   \frac{1}{1-\varpi}=\frac{pq}{2pq-1}.
\end{equation}
By substituting (\ref{eq:2}) and  (\ref{eq:3}) into  (\ref{eq:1}), one deduces $\E{\tau_j} \le  \frac{pq}{2pq-1}$ and hence the proof is completed. 
\end{proof}

We now introduce a counting process $N(j)$ given by the number of renewals that occur before time $j$:
$$
N(j):= \max \left\{ i : A_i \le j \right\}.
$$ 
We also consider the sequence of random variables defined by $Y_0=\Phi_0$ and 
 $$Y_j =\Phi_{\min(j,T_{\epsilon})} + \sigma \sum_{k=0}^{\min(j,T_{\epsilon}) -1} \left(
 \frac{1}{\Mu_k^2 \|r(X_k)\|^{\ipow}} \right)
 \quad \mbox{for all $j \ge 1$.}$$

The definition of $\{Y_j\}$ is our second and main distinction from the analysis of stochastic 
trust region (see~\cite[Lemma 2.2]{JBlanchet_CCartis_MMenickelly_KScheinberg_2019}). 

\begin{lemma} 
 \label{lm:E_N_T_eps}
 Under the assumptions of Proposition \ref{lm:decrease:phi}. Let Assumption  \ref{assum:bnds:mu&Phi} hold. One has,
$$
\E{N(T_{\epsilon})} \le \frac{\Phi_0 \kappa_r^{\ipow}}{\sigma}\Mu_\epsilon^2.
$$
\end{lemma}

\begin{proof} 
Note that $Y_{j} $ defines a supermartingale with respect to $\mathcal{F}_{j-1}^{M \cdot R}$. Indeed, if $j< T_{\epsilon}$, then using Proposition \ref{lm:decrease:phi} one has,
\begin{eqnarray*}
\E{Y_{j+1}| \mathcal{F}_{j-1}^{M \cdot R}} 
&= &
\E{ \Phi_{j+1}  | \mathcal{F}_{j-1}^{M \cdot R}} 
+ \E{\sigma \sum_{k=0}^{j}\frac{1}{\Mu_k^2\,\|r(X_k)\|^{\ipow}} 
\middle| \mathcal{F}_{j-1}^{M \cdot R}} \\
 & \le & \Phi_{j}  - \sigma \left(\frac{1}{\Mu_j^2 \|r(X_j)\|^{\ipow}}\right)+ \sigma \sum_{k=0}^{j}\left(\frac{1}{\Mu_k^2 \|r(X_k)\|^{\ipow}}\right)\\
 &= & \Phi_{j}  +  \sigma \sum_{k=0}^{j -1}\left(\frac{1}{\Mu_k^2\|r(X_k)\|^{\ipow}}\right) =Y_{j}.
\end{eqnarray*}
If $j\ge T_{\epsilon}$, one has $Y_{j+1}= Y_j$ and thus $\E{Y_{j+1}| \mathcal{F}_{j-1}^{M \cdot R}}= Y_j$.
Using Assumption \ref{assum:bnds:mu&Phi}, one has for all $j\ge T_{\epsilon}$, $|Y_{j} | = |Y_{T_{\epsilon}}| \le \left(\Phi_{\max} + \frac{ (T_{\epsilon}+1) \sigma}{\mu^2_{\min}}\right)\epsilon_p^{\frac{1-2^{\icc+1}}{2^\icc}}$. Hence, since $T_{\epsilon}$ is bounded, $Y_{j}$ is also bounded. Because $T_{\epsilon}$ is a stopping time, the optimal stopping theorem~\cite[Theorem 6.4.1]{SMRoss_1983} for supermartingales applies, and we have
$$
 \E{ Y_{T_{\epsilon}}} \le \E{Y_0}.
$$
Hence, 
\begin{eqnarray} \label{ineq:OST}
\sigma \E{ \sum_{k=0}^{T_{\epsilon}} \frac{1}{\Mu_k^2 \|r(X_k)\|^{\ipow}} } \le  \E{ Y_{T_{\epsilon}}} &\le & \E{Y_0} =  \Phi_{0}.
\end{eqnarray}
By the definition of the counting process $N(T_{\epsilon})$, since the renewal times $A_i$ (which satisfy $\Mu_{A_i}\le \Mu_\epsilon$) are a subset of the iterations $0, 1,\ldots, T_{\epsilon}$, one has
\[
	\sum_{k=0}^{T_{\epsilon}} \frac{1}{\Mu_k^2 \|r(X_k)\|^{\ipow}} 
	\ge \sum_{k=0}^{T_{\epsilon}} \frac{1}{\Mu_k^2 \kappa_r^{\ipow}}
	\ge \frac{N(T_{\epsilon})}{\Mu_\epsilon^2 \kappa_r^{\ipow}}
\]
Inserting the latter inequality in (\ref{ineq:OST}), one gets
\begin{eqnarray*}
 \E{N(T_{\epsilon}) } &\le &  \frac{\Phi_{0}\kappa_r^{\ipow}}{\sigma}\Mu_\epsilon^2,
\end{eqnarray*}
which concludes the proof.
\end{proof}

Using Wald's equation~\cite[Corollary 6.2.3]{SMRoss_1983}, we can finally obtain a bound on the 
expected value of $T_{\epsilon}$.

 \begin{theorem} \label{theo:wcc}
 Under the assumptions of Proposition \ref{lm:decrease:phi} as well as Assumption~\ref{assum:bnds:mu&Phi},
 one has
$$
\E{T_{\epsilon}} \le \frac{pq}{2pq-1} \left( \frac{\Phi_0\kappa_r^{\ipow}}{\sigma} \Mu_\epsilon^2 +1 \right) -1 \le \frac{pq}{2pq-1} \left(\kappa_s \kappa_r^{\ipow} \epsilon_d^{-2}\epsilon_p^{-\frac{2^{\icc+1}-1}{2^{\icc-1}}}+1\right) -1.
$$
where $\kappa_s = \frac{\tau f(x_0) +(1-\tau) \mu^{-2}_0}{\tfrac{1}{4}(1-\tau)\left(1-\tfrac{1}{\lambda^2}\right)} \zeta^2$, 
	$\tau \in (0,1)$ \revised{satisfies~\eqref{eq:taumuseries} for a value $\zeta$ such 
	that~\eqref{eq:zetamuseries} holds.}
\end{theorem}

\begin{proof}
First note that the renewal process $A_{N(T_\epsilon)+1} = \sum^{N(T_\epsilon)}_{i=0} \tau_i$ where $\tau_i$ defines independent inter-arrival times. Moreover, since the probabilities $p$ and $q$ satisfy (\ref{eq:condprobaspq1}), one has $pq>1/2$ and hence, by applying Lemma \ref{lm:E_tau}, for all $i=1, \ldots, N(T_\epsilon)$ one has $\E{\tau_i} \le \frac{pq}{2pq-1}< +\infty$. Thus, by Wald's equation \cite[Corollary 6.2.3]{SMRoss_1983}, one gets,
$$
\E{A_{N(T_\epsilon)+1}} = \E{\tau_1} \E{N(T_{\epsilon})+1} \le  \frac{pq}{2pq-1} \E{N(T_{\epsilon})+1}  .
$$
By the definition of $N(T_\epsilon)$ one has $A_{N(T_{\epsilon}) +1} > T_{\epsilon}$, hence using Lemma \ref{lm:E_N_T_eps} one gets
$$
\E{T_{\epsilon}} \le \E{A_{N(T_\epsilon)+1}} -1 \le \frac{pq}{2pq-1} \left( \frac{\Phi_0\kappa_r^{\ipow}}{\sigma} \Mu_\epsilon^2  +1  \right) -1,
$$
which establishes the result by definition of $\kappa_s$, $\Phi_0$ and $\Mu_\epsilon$.
\end{proof}
We now comment on the complexity orders appearing in Theorem~\ref{theo:wcc}.
When $\icc=-1$, the result of the theorem matches the bounds derived for 
stochastic trust-region methods~\cite{JBlanchet_CCartis_MMenickelly_KScheinberg_2019}. 
For $\icc=0$, the order becomes $\epsilon_d^{-2} \epsilon_p^{-2}$ in expectation 
for the classical scaled optimality criterion. In addition, as 
$\icc \rightarrow \infty$, the bound of Theorem~\ref{theo:wcc} asymptotically 
becomes of order $\epsilon_d^{-2} \epsilon_p^{-4}$ for 
the scaled criterion $\tfrac{\|J(x)^\top r(x)\|}{\|r(x)\|^2} \le \epsilon_d$ or 
$\|r(x)\| \le \epsilon_p$.
Note that if we set $\epsilon_p=\epsilon \epsilon_d^{-1}$ for some 
$\epsilon \in (0,\epsilon_d)$, we recover a bound in $\epsilon^{-2}$ for the 
case $\icc=0$, matching the order obtained asymptotically by Gould et 
al.~\cite{NIMGould_TRees_JAScott_2019}.

\section{Applications}
\label{sec:applis}

\revised{In this section, we discuss the relevance of our approach for several 
formulations arising in inverse problems, data assimilation and machine 
learning.}

\subsection{Ensemble methods for inverse problems and data assimilation}
\label{subsec:applis:enkf}

\revised{We describe the problem of interest using the terminology of inverse problems. 
Let $x^*$ be a ground truth vector and $y$ be a vector of observations such that
\begin{equation} 
\label{eq:inversepb}
	y  = \mathcal{H}(x^*) + \zeta,
\end{equation}
where $\mathcal{H}$ is the observational operator mapping the unknowns from 
the state space to the observation space and 
$\zeta \sim \mathcal{N}(0,R)$ is a Gaussian noise vector with given covariance matrix $R$. The maximum likelihood estimator for the vector $x^*$ can then be obtained by solving:
\begin{equation} \label{eq:likelihood}
\displaystyle \min_{x}  
\frac{1}{2}\left(\|y - \mathcal{H}(x)\|^2_{R^{-1}}\right).
\end{equation}
In inverse problems, the formulation~(\ref{eq:likelihood}) is typically ill-posed, 
and a regularized version of this problem is used instead. 
The formulation of Iglesias et al.~\cite{Iglesias_2013} incorporates
prior knowledge of $x$ in the form of a 
finite dimensional space where problem (\ref{eq:likelihood}) is solved.
This approach bears strong connection with data assimilation, in which a prior knowledge on the state  
of the form $x = b + \xi $ is typically assumed, 
with $b$ being a fixed prior (or a background) on $x$ and $\xi \sim  \mathcal{N}(0,B^{\infty})$ 
being a Gaussian noise vector with unknown covariance matrix $B^{\infty}$ (see for instance \cite{Iglesias_2016, Iglesias_2020, Chada_2020}). 
With such a prior,
problem (\ref{eq:likelihood}) is then replaced by
 \begin{equation} \label{4dvarsc-c}
\displaystyle \min_{x}  
\frac{1}{2}\left(\|y - \mathcal{H}(x)\|^2_{R^{-1}}+  \|x - b\|^2_{(B^\infty)^{-1}}\right),
\end{equation}
which is a maximum a posteriori estimation problem. The latter formulation matches the 
structure of problem~\eqref{eq:mainpb} with
$r(x):=\left(\begin{array}{c} (B^{\infty})^{-1/2}(x-x_b) \\ 
R^{-1/2}(y-\mathcal{H}(x)) \end{array} \right)$. 
}


\revised{In both data assimilation and inverse problems, an important issue with 
formulation~\eqref{4dvarsc-c} is the estimation of the covariance matrix 
$B^{\infty}$. One possible approach to address this issue 
consists in estimating $B^{\infty}$ using a random approximation via an ensemble.
Assuming that an initial ensemble of $\hat{n}$ elements 
$\left\{\hat x^k\right\}_{k=1}^{\hat{n}}$ is available, the matrix 
$B^{\infty}$ is approximated by the empirical covariance matrix of 
the ensemble, given by
\begin{eqnarray}\label{emp:cov:mat}
B^{\hat{n}} &:=& \frac{1}{\hat{n}-1} 
\sum_{k=1}^{\hat{n}} \left(\hat x^k - b\right) \left(\hat x^k - b\right)^{\top}.
\end{eqnarray}
By assuming that the initial ensemble is sampled from the Gaussian distribution with mean $b$ and unknown covariance matrix $B^{\infty}$,  
 then the matrix $B^{\hat{n}}$ follows a Wishart distribution~\cite{Wishart_1928}. 
Thus, if $\hat{n} \ge n+1$, then $B^{\hat{n}}$ is nonsingular with 
probability one, $(B^{\hat{n}})^{-1}$ follows an inverse Wishart distribution, and
\begin{eqnarray}\label{eq:Wishart:expected}
\E {B^{\hat{n}}} = B^\infty \text{ and }  \E  {(B^{\hat{n}})^{-1}} 
&= &\frac{\hat{n}-1}{\hat{n}-1-n} {(B^\infty)^{-1}}.
\end{eqnarray}
Provided $\hat{n}$ is large enough relative to $n$, the empirical covariance matrix 
$B^{\hat{n}}$ can be assumed to be non-singular and $\E{(B^{\hat{n}})^{-1}}$ 
is very close to ${(B^\infty)^{-1}}$. 
Using $B^{\hat{n}}$ to approximate $B^{\infty}$, a noisy approximation of the 
objective function in~\eqref{4dvarsc-c} is given by
\begin{equation} \label{4dvarsc-cc-} 
{\frac{1}{2} \left(\|y - \mathcal{H}(x)\|^2_{R^{-1}}  + \|x - b\|^2_{(B^{\hat{n}})^{-1}}
\right)}.
\end{equation}
%
}

\revised{
Iglesias et al.~\cite{Iglesias_2013, Iglesias_2016} proposed approximately solving
~(\ref{4dvarsc-cc-}) using the Ensemble Kalman Inversion 
(EKI) method~\cite{ATarantola_2005}, an adaptation of
the ensemble Kalman filter (EnKF) used in data 
assimilation~\cite{GEvensen_1994,GEvensen_2009} to inverse problems. 
The TEKI (Tikhonov-EKI) approach of Chada et al.~\cite{Chada_2020} adds an 
extra regularization of the form $\frac{1}{2} \gamma \|x\|^2$  to the 
formulation~(\ref{4dvarsc-cc-}), note the resemblance to the 
Levenberg-Marquardt approach. More recently, Iglesias and 
Yang~\cite{Iglesias_2020} introduced a new adaptive regularization
strategy for EKI that comes even closer to our work, with the most significant
distinction being that the subproblems 
considered in the EKI framework are typically nonlinear. By linearizing the 
operator $\mathcal{H}$ as in Gauss-Newton techniques, one would obtain 
an algorithm quite close in spirit to Algorithm~\ref{alg:LM}, and our 
analysis could help in endowing that method with global convergence and 
complexity guarantees even in the presence of noise in the objective value.
}

\revised{
Despite this connection, we acknowledge that our setup (and in particular, 
satisfaction of the properties~\eqref{eq:Wishart:expected}) requires 
significantly more samples than what is usually used in practical ensemble 
methods. Furthermore, our analysis relies on conditional independence between 
the iterations, which would amount to generate new ensembles at every iteration. 
Such approaches would likely be impractical in data assimilation and inverse 
problems, yet we believe that the links between our results could be informative 
in developing new variants of the EKI approach with global guarantees.
}

\subsection{Subsampling approaches}
\label{subsec:applis:ML}

\revised{Least-squares problems play a prominent role in data analysis, as the 
least-squares loss is a common way to assess the accuracy of a machine learning 
model. In such a setup, the residual function can be written as 
$r(x) = [ y_i - h(z_i,x)]_{i=1}^m$, where $\{(z_i,y_i)\}_{i=1}^m$ are data 
points and $h$ is a learning model parameterized by $x$, such as the output 
of a neural network. In modern data science applications, the number of data 
points is usually quite large, but the data is assumed to follow a known 
distribution form. As a result, one can rely on using subsamples of the data points 
to estimate the function values and the gradient: such a paradigm is at the 
essence of stochastic gradient methods~\cite{LBottou_FECurtis_JNocedal_2018}.}

\revised{In this context, it is possible to guarantee the probabilistic 
properties defined in Section~\ref{sec:LMprobaprop} with a sufficiently large 
number of samples. Following recent work in optimization and 
subsampling~\cite{EBergou_YDiouane_VKunc_VKungurtsev_CWRoyer_2021}, we know that 
$\mathcal{O}\left(\ln\left(\tfrac{1}{\delta}\right)\Mu_j^4\right)$ samples are needed to get 
accurate function estimates with probability $1-\delta$, while 
$\mathcal{O}\left(\ln\left(\tfrac{1}{\delta}\right)\Mu_j^2\right)$ samples can be used to build 
a gradient estimate that is accurate with probability $1-\delta$. Therefore, 
in order to satisfy Assumptions~\ref{assum:models} and~\ref{assum:festimates}, 
we require 
\[
	\mathcal{O}\left(\ln\left(\tfrac{1}{1-p}\right)\max\{\Mu_j^4,\Mu_j^2\}\right) 
	+ \mathcal{O}\left(\ln\left(\tfrac{1}{1-q}\right)\Mu_j^4 \right)
\]
samples per iteration.
Note that these rates naturally extend to the stochastic 
case~\cite{JLarson_SCBillups_2016}.}
\section{Conclusion} \label{sec:conclusion}

We proposed a stochastic Levenberg-Marquardt method
to solve nonlinear least-squares problems wherein the objective function and 
its gradient are subject to noise and can only be computed accurately within a 
certain probability. By employing a scaling formula for the Levenberg-Marquardt 
parameter, we leveraged the link between our approach and a 
trust-region-type framework to obtain complexity bounds in expectation for our 
method. Our guarantees are based upon a scaled gradient criterion that 
exploits the least-squares structures of our problem, and generalizes 
previously proposed metrics. \revised{We have illustrated how our working assumptions can hold in the context of
 inverse problems and machine learning where the noise in the objective arises from empirical covariance estimators and subsampling techniques, respectively.}

The study of the performance of our approach when applied to large-scale data 
assimilation problems is a natural continuation of the present work, that poses 
additional challenges. Indeed, in practical situations, a single ensemble may 
be used at each iteration, which would introduce correlations between the  
model and the estimates. In addition, the ensemble size might be significantly 
smaller than the dimension of the state space, which can jeopardize the quality 
of the ensemble approximations. Extending our analysis to these settings raises 
a number of theoretical issues, which we plan to investigate so as to hew our 
method closer to standard practice.

\section*{Acknowledgments}

The authors would like to thank Matt Menickelly and Katya Scheinberg for 
useful discussions regarding the STORM algorithm.
\revised{The authors are also grateful to the associate editor and two 
anonymous referees for the careful reading of the paper and their suggestions.}

\bibliographystyle{siamplain}
\bibliography{LM-stochastic}

\newpage
\begin{appendices}

\section{Proof of Theorem~\ref{theo:cvmuseries}}
\label{sec:prooftheomuseries}

\begin{proof}[Proof of Theorem~\ref{theo:cvmuseries}]
	Consider a realization of Algorithm~\ref{alg:LM}, and let $\phi_j$ be the 
	corresponding realization of $\Phi_j$. If $j$ is the index of a successful 
	iteration, then $x_{j+1} = x_j + s_j$, and 
	$\mu_{j+1} \ge \tfrac{\mu_j}{\lambda}$. One thus has:
	\begin{equation} \label{eq:phidecsuccits}
		\phi_{j+1} - \phi_j \; \le \; 
		\tau\left(\|\revised{r(x_{j+1})}\|^{1/2^\icc}-\|\revised{r(x_j)}\|^{1/2^\icc} \right) +   
		(1-\tau)\frac{\lambda^{2} - 1}{\mu_j^2 \|r_j\|^{\ipow}}.
	\end{equation}
	If $j$ is the index of an unsuccessful iteration, $x_{j+1}=x_j$ and 
	$\mu_{j+1}=\lambda\mu_j$, leading to 
	\begin{equation} \label{eq:phidecunsuccits}
		\phi_{j+1} - \phi_j \; = \; (1-\tau)\left(\frac{1}{\lambda^2}-1\right)
		\frac{1}{\mu_j^2 \|\revised{r(x_j)}\|^{\ipow}} < 0.
	\end{equation}
	For both types of iterations, we will consider four possible outcomes, 
	involving the quality of the model and that of the estimates. 
	
	\medskip
	
	\emph{Case 1: $\|J(x_j)^\top r(x_j)\| \ge \tfrac{\zeta}{\mu_j}$.}

	\begin{enumerate}
		\item \textbf{Both $m_j$ and $(r^0_j,r^s_j)$ are accurate.} Since we 
		are in Case 1,
		\[
			\left\| J(x_j)^\top r(x_j)\right\| 
			\ge \left( \kappa_{eg} + \kappa_{\mu g} \right)
			\frac{1}{\mu_j}.
		\]
		Because the model is $(\kappa_{ef},\kappa_{eg})$-first-order accurate, this 
		implies 
		\[
			\|J_{m_j}^\top r_{m_j}\| \ge \|J(x_j)^\top r(x_j)\|-\frac{\kappa_{eg}}{\mu_j} 
			\ge (\zeta-\kappa_{eg})\frac{1}{\mu_j} \ge \frac{\kappa_{\mu g}}{\mu_j}
		\]
		so~\eqref{eq:condmodelgraditsucc} holds; since the estimates 
		are also accurate, the iteration is \revised{necessarily} successful by 
		Lemma~\ref{lemma:condmodelgraditsucc}. Moreover, 
		\[
			\|J(x_j)^\top r(x_j)\| \; \ge \; \frac{\zeta}{\mu_j} 
			\; \ge \; \left(\kappa_{eg} + \max\left\{\kappa_{J_m}^2,
			\frac{8(\kappa_{ef}+\kappa_{efs})}{\eta_1 \theta_{fcd}}\right\} 
			\right)\frac{1}{\mu_j} ,
		\]
		so the condition~\eqref{eq:mudecreasegoodmodeltruegrad} is 
		satisfied, and by Lemma~\ref{lemma:decreasegoodmodeltruegrad}, we 
		can guarantee a decrease on the function value. More precisely,
		\begin{equation*}
			\|r(x_{j+1})\|^{1/2^\icc} - \|r(x_j)\|^{1/2^\icc} \; \le \; 
			-\frac{C_1}{2^{\icc+1} \frac{\|J(x_j)^\top r(x_j)\|}
			{\|r(x_j)\|^{\ipow}\mu_j}},
		\end{equation*}
		leading to
		\begin{eqnarray} \label{eq:case1a}
			\varphi_{j+1} - \varphi_j &\le &-\tau \frac{C_1}{2^{\icc+1}} 
			\frac{\|J(x_j)^\top r(x_j)\|}{\|r(x_j)\|^{\ipow}\mu_j}			
			+(1-\tau)(\lambda^{2}-1)\frac{1}{\mu_j^2\|r(x_j)\|^{\ipow}}.
		\end{eqnarray}
		\item \textbf{Only $m_j$ is accurate.} The decrease formula of 
		Lemma~\ref{lemma:decreasegoodmodeltruegrad}
		is valid in \revised{this} case: therefore, if the iteration is successful, 
		then~\eqref{eq:case1a} holds, and we have
		\begin{eqnarray*} 
			\varphi_{j+1} - \varphi_j &\le &-\tau \frac{C_1}{2^{\icc+1}} 
			\frac{\|J(x_j)^\top r(x_j)\|}{\|r(x_j)\|^{\ipow}\mu_j}			
			+(1-\tau)(\lambda^{2}-1)\frac{1}{\mu_j^2\|r(x_j)\|^{\ipow}} \\
			&\le &\frac{-\tau \tfrac{C_1}{2^{\icc+1}} \zeta + 
			(1-\tau)(\lambda^{2}-1)}{\mu_j^2\|r(x_j)\|^{\ipow}} 
			\nonumber \\
			&\le &\frac{-\tau \tfrac{C_1}{2^{\icc+2}} \zeta + 
			(1-\tau)(\lambda^{2}-1)}{\mu_j^2\|r(x_j)\|^{\ipow}} 
			\nonumber \\
			&< &(1-\tau)\left(\frac{1}{\lambda^2}-1\right) 
			\frac{1}{\mu_j^2\|r(x_j)\|^{\ipow}},
			\nonumber
		\end{eqnarray*}		
		using~\eqref{eq:taumuseries} to obtain the last inequality. Therefore, 
		\eqref{eq:phidecunsuccits} holds when the iteration is successful. 
		From the beginning of the proof, we know that \eqref{eq:phidecunsuccits} 
		also holds if the iteration is unsuccessful.
		\item \textbf{Only $(r^0_j,r^s_j)$ is accurate.} If the 
		iteration is unsuccessful, then~\eqref{eq:phidecunsuccits} is 
		satisfied. Otherwise, we can apply Lemma~\ref{lemma:decreasegoodestim} 
		and have a guarantee of decrease in the case of a successful iteration, 
		namely,
		\[
			\|r(x_{j+1})\|^{1/2^\icc} - \|r(x_j)\|^{1/2^\icc} \le -\frac{C_2}{2^{\icc+1}} 
			\frac{1}{\mu_j^2 \|r(x_j)\|^{\ipow}},
		\]
		hence
		\begin{eqnarray*} 
			\varphi_{j+1} - \varphi_j &\le &\left[-\tau \frac{C_2}{2^{\icc+1}} + 
			(1-\tau)(\lambda^{2}-1)\right]\frac{1}{\mu_j^2\|r(x_j)\|^{\ipow}} \\
			&\le &(1-\tau)\left(\frac{1}{\lambda^{2}}-1\right)
			\frac{1}{\mu_j^2\|r(x_j)\|^{\ipow}},
		\end{eqnarray*}
		where we used~\eqref{eq:taumuseries} to obtain the last inequality. 
		Thus~\eqref{eq:phidecunsuccits} also holds if the iteration is successful.
		\item \textbf{Both $m_j$ and $(r^0_j,r^s_j)$ are inaccurate.} If the 
		iteration is unsuccessful, then~\eqref{eq:phidecunsuccits} holds. Suppose 
		now that the iteration is successful: by considering a Taylor expansion of 
		$f(x_j+s_j)$ around $x_j$, we know that the possible increase in the step 
		is bounded above by:
		\begin{eqnarray*}
			f(x_j+s_j) - f(x_j) &\le &\nabla f(x_j)^\top s_j + \frac{\revised{\nu}}{2}\|s_j\|^2 \\
			&= &\left[J(x_j)^\top r(x_j)\right]^\top s_j + \frac{\revised{\nu}}{2}\|s_j\|^2 \\
			&\le &\|J(x_j)^\top r(x_j)\| \|s_j\| + \frac{\revised{\nu}}{2}\|s_j\|^2 \\
			&\le &\|J(x_j)^\top r(x_j)\| \|s_j\| + \revised{\nu}\frac{\|s_j\|}{\mu_j} \\
			&\le &\left(1+\frac{\revised{\nu}}{\zeta}\right)\|J(x_j)^\top r(x_j)\|\|s_j\| \\
			&\le &2\left(1+\frac{\revised{\nu}}{\zeta}\right)\frac{\|J(x_j)^\top r(x_j)\|}{\mu_j}.
		\end{eqnarray*}
		We thus have
		\[
			\|r(x_j+s_j)\|^2 - \|r(x_j)\|^2 \le 4\left(1+\frac{\revised{\nu}}{\zeta}\right)
			\frac{\|J(x_j)^\top r(x_j)\|}{\mu_j}.
		\]		
		Suppose that $\|r(x_j+s_j)\| > \|r(x_j)\|$. Applying 
		Lemma~\ref{lemma:oursab} to the above equation yields:
		{\small
		\[
			\|r(x_j+s_j)\|^{1/2^\icc} - \|r(x_j)\|^{1/2^\icc} \le 
			\frac{4(1+\frac{\revised{\nu}}{\zeta})}{\mu_j}
			\frac{\|J(x_j)^\top r(x_j)\|}{\|r(x_j+s_j)\|^{\ipow}} \le 
			\frac{4\left(1+\frac{\revised{\nu}}{\zeta}\right)}{\mu_j}
			\frac{\|J(x_j)^\top r(x_j)\|}{\|r(x_j)\|^{\ipow}},
		\]
		}
		and note that this relation still holds when $\|r(x_j+s_j\|\le \|r(x_j)\|$ 
		as the left-hand side is negative in that case.
		We thus obtain the following bound on the change in $\phi$:
		\begin{equation} \label{eq:case1d}
			\phi_{j+1} - \phi_j \; \le \; \tau \frac{C_3}{2^{\icc+1}} 
			\frac{\|J(x_j)^\top r(x_j)\|}{\|r(x_j)\|^{\ipow}} \frac{1}{\mu_j}
			+ (1-\tau)(\lambda^{2}-1)\frac{1}{\mu_j^2},
		\end{equation}
		\revised{with $C_3  = 4\left(1+\tfrac{\revised{\nu}}{\zeta}\right)$}. The 
		right-hand side of~\eqref{eq:case1d} is nonnegative and larger than 
		that of~\eqref{eq:phidecunsuccits}, thus the bound~\eqref{eq:case1d}
		holds when the iteration is successful and when it is unsuccessful.
	\end{enumerate}	
	Summarizing the four cases, we have that the bound~\eqref{eq:case1a} on 
	$\Phi_{j+1}-\Phi_j$ holds in case 1-1, the 
	bound~\eqref{eq:phidecunsuccits} holds in both cases 1-2 and 1-3, and the 
	bound~\eqref{eq:case1d} holds in case 1-4.
	Putting those together with their associated probability of  
	occurrence, we obtain:
	{\small 
	\begin{eqnarray*}
	& & 
	\E{\Phi_{j+1}-\Phi_j | \mathcal{F}_{j-1}^{M \cdot R}\cap E^0_j \cap 
	\left\{\|J(X_j)^\top r(X_j)\| \ge \frac{\zeta}{\Mu_j} \right\}} \\ 
	&\le & 
	pq \left[ -\tau \frac{C_1}{2^{\icc+1}}
	\frac{\|J(X_j)^\top r(X_j)\|}{\|r(X_j)\|^{\ipow}} \frac{1}{\Mu_j} + 
	(1-\tau)(\lambda^{2}-1)\frac{1}{\Mu_j^2\|r(X_j)\|^{\ipow}}\right] \\
	& &
	+[p(1-q)+(1-p)q]\left[(1-\tau)\left(\frac{1}{\lambda^2}-1\right)
	\frac{1}{\Mu_j^2}\right] \\
	& &
	+(1-p)(1-q)\left[\tau \frac{C_3}{2^{\icc+1}} 
	\frac{\|J(X_j)^\top r(X_j)\|}{\|r(X_j)\|^{\frac{2^{\icc+1}-1}{2^{\icc+1}}}} 
	\frac{1}{\Mu_j}
	+ (1-\tau)(\lambda^{2}-1)\frac{1}{\Mu_j^2\|r(X_j)\|^{\ipow}}\right] \\
	&= &
	\frac{-C_1 pq + (1-p)(1-q)C_3}{2^{\icc+1}}\tau
	\frac{\|J(X_j)^\top r(X_j)\|}{\|r(X_j)\|^{\frac{2^{\icc+1}-1}{2^{\icc+1}}}}
	\frac{1}{\Mu_j^2} \\
	& &
	+[pq-\tfrac{1}{\lambda^2}(p(1-q)+(1-p)q)+(1-p)(1-q)] (1-\tau)(\lambda^{2}-1)
	\frac{1}{\Mu_j^2\|r(X_j)\|^{\ipow}} \\
	&\le &
	\frac{-C_1 pq + (1-p)(1-q)C_3}{2^{\icc+1}}\tau
	\frac{\|J(X_j)^\top r(X_j)\|}{\|r(X_j)\|^{\frac{2^{\icc+1}-1}{2^{\icc+1}}}}
	\frac{1}{\Mu_j^2} 
	+ (1-\tau)(\lambda^{2}-1)\frac{1}{\Mu_j^2\|r(X_j)\|^{\ipow}},
	\end{eqnarray*}
	}   

	where the last line uses 
	\[
		pq-\tfrac{1}{\lambda^2}(p(1-q)+(1-p)q)+(1-p)(1-q) \le (p + (1-p))(q+(1-q)) = 1.
	\]
	Suppose $p$ and $q$ are chosen such that
	\begin{equation}  \label{eq:condpqmuseries1}
		\frac{pq-1/2}{(1-p)(1-q)} \ge \frac{C_3}{C_1},
	\end{equation}	
	holds. Then, one has by combining~\eqref{eq:condpqmuseries1} 
	and~\eqref{eq:taumuseries}:
	\begin{equation} \label{eq:boundC1}
		\frac{-C_1 pq + (1-p)(1-q)C_3}{2^{\icc+1}} 
		\le -\frac{1}{2^{\icc+2}} C_1 
		\le -2\frac{(1-\tau)(\lambda^{2}-1)}{\tau\zeta}.
	\end{equation}
	On the other hand, since $\|J(X_j)^\top r(X_j)\| \ge \zeta/\Mu_j$, we have:
	\begin{equation*}
		(1-\tau)(\lambda^{2}-1)\frac{1}{\Mu_j^2 \|r(X_j)\|^{\ipow}} 
		\; \le \; -\frac{1}{2}[-C_1 pq + (1-p)(1-q)C_3]\tau
		\frac{\|J(X_j)^\top r(X_j)\|}{\|r(X_j)\|^{\ipow}}\frac{1}{\Mu_j}.
	\end{equation*}
	This leads to
	\begin{eqnarray*}
		\E{\Phi_{j+1}-\Phi_j | \mathcal{F}_{j-1}^{M \cdot R} \cap E^0_j \cap 
		\left\{\|J(X_j)^\top r(X_j)\| \ge \frac{\zeta}{\Mu_j}\right\}} 
		&\le &-\frac{1}{4}C_1\tau\frac{\|J(X_j)^\top r(X_j)\|}{\|r(X_j)\|^{\ipow}}
		\frac{1}{\Mu_j} \\
		&\le &-\frac{1}{4}C_1\tau\zeta\frac{1}{\Mu_j^2\|r(X_j)\|^{\ipow}},
	\end{eqnarray*}
	which, using~\eqref{eq:boundC1}, finally gives:
	{\small{
	\begin{eqnarray} \label{eq:case1}
		\E{\Phi_{j+1}-\Phi_j | \mathcal{F}_{j-1}^{M \cdot R} \cap E^0_j \cap  
		\left\{\|J(X_j)^\top r(X_j)\| \ge \frac{\zeta}{\Mu_j} \right\}} 
		&\le &-\frac{(1-\tau)\left(\lambda^{2}-1 \right)}{4}
		\frac{1}{\Mu_j^2\|r(X_j)\|^{\ipow}} \nonumber \\
		&\le &-\frac{(1-\tau)\left(1-\tfrac{1}{\lambda^2}\right)}{4}
		\frac{1}{\Mu_j^2\|r(X_j)\|^{\ipow}}.
	\end{eqnarray}
	}}

	\medskip
	
	\emph{Case 2: $\|J(x_j)^\top r(x_j)\| < \tfrac{\zeta}{\mu_j}$.}

	Whenever $\|J_{m_j}^\top r_{m_j}\| < \tfrac{\eta_2}{\mu_j}$, the iteration 
	is necessarily unsuccessful and~\eqref{eq:phidecunsuccits} holds. We thus 
	assume in what follows that 
	$\|J_{m_j}^\top r_{m_j}\| \ge \tfrac{\eta_2}{\mu_j}$, and consider again four 
	cases.	
	\begin{enumerate}
		\item \textbf{Both $m_j$ and $(r^0_j,r^s_j)$ are accurate.} Unlike case 1-a), 
		it is now possible for the iteration to be unsuccessful: in that case, we 
		have~\eqref{eq:phidecunsuccits}. Otherwise, if the iteration is 
		successful, then we can use the result from Lemma~\ref{lemma:decreasegoodestim}, 
		and we have:
		\[
			\|r(x_{j+1})\|^{1/2^\icc} - \|r(x_j)\|^{1/2^\icc} \le -\frac{C_2}{2^{\icc+1}}
			\frac{1}{\mu_j^2\,\|r(x_j)\|^{\ipow}}.
		\]
		We can thus apply the same reasoning than in case 1-3, which implies 
		that~\eqref{eq:phidecunsuccits} also holds when the iteration is successful.
		\item \textbf{Only $m_j$ is accurate.} If the iteration is unsuccessful, it 
		is clear that~\eqref{eq:phidecunsuccits} holds. Otherwise, using 
		$\eta_2 \ge \kappa_{J_m}^2$ that arises from~\eqref{eq:condeta2cvmuseries} 
		and applying the same argument as in the proof of 
		Lemma~\ref{lemma:decreasegoodestim}, we have 
		$\|J_{m_j}^\top r_{m_j}\| \ge \tfrac{\eta_2}{\mu_j} \ge 
		\tfrac{\kappa_{J_m}^2}{\mu_j}$, leading to
		\begin{eqnarray*}
			 m_j(x_j) - m_j(x_j+s_j) 
			 &\ge &\frac{\theta_{fcd}}{2}\frac{\|J_{m_j}^\top r_{m_j}\|^2}
			 {\kappa_{J_m}^2+ \mu_j \|J_{m_j}^\top r_{m_j}\|} \\
			 &\ge &\frac{\theta_{fcd}}{2}	
			 \frac{\|J_{m_j}^\top r_{m_j}\|^2}{2\mu_j \|J_{m_j}^\top r_{m_j}\|} 
			 = \frac{\theta_{fcd}}{4} \frac{\|J_{m_j}^\top r_{m_j}\|}{\mu_j} 
			 \ge \frac{\eta_2\theta_{fcd}}{4}
			\frac{1}{\mu_j^2}.
		\end{eqnarray*}
		Since the model is $(\kappa_{ef},\kappa_{eg})$-first-order accurate, the 
		function variation satisfies:
		\begin{eqnarray*}
			f(x_j) - f(x_j+s_j) &= 
			&f(x_j)-m_j(x_j)+m_j(x_j)-m_j(x_j+s_j)+m_j(x_j+s_j)-f(x_j+s_j) \\
			&\ge &-\frac{\kappa_{ef}}{\mu_j^2} + \frac{\eta_2\theta_{fcd}}{4}
			\frac{1}{\mu_j^2} - \frac{\kappa_{efs}}{\mu_j^2} \\
			&\ge &\left(\frac{\eta_2\theta_{fcd}}{4}-(\kappa_{ef}+\kappa_{efs})\right)
			\frac{1}{\mu_j^2} 
			\; \ge \; \frac{\kappa_{ef}+\kappa_{efs}}{2\mu_j^2},
		\end{eqnarray*}
		where the last line comes from~\eqref{eq:condeta2cvmuseries}.\\ 
		As a result, we have 
		$\|r(x_j+s_j)\|^2 - \|r(x_j)\|^2 \le - \tfrac{\kappa_{ef} + \kappa_{efs}}{\mu_j^2}$.
		Applying \revised{Lemma~\ref{lemma:oursab}} then gives:
		\[
			\|r(x_j+s_j)\|^{1/2^\icc} - \|r(x_j)\|^{1/2^\icc} \; \le \; 
			-\frac{\kappa_{ef}+\kappa_{efs}}{2^{\icc+1}}\frac{1}{\mu_j^2 \|r(x_j)\|^{\ipow}},
		\]
		and this leads to
		\begin{eqnarray} \label{eq:case2b}
			\phi_{j+1} - \phi_j &\le &\left[ - \tau\frac{\kappa_{ef}+\kappa_{efs}}{2^{\icc+1}}\frac{1}{\mu_j^2 \|r(x_j)\|^{\ipow}}+ 
			(1-\tau)(\lambda^{2}-1)\right]\frac{1}{\mu_j^2} \nonumber \\ 
			\phi_{j+1} - \phi_j &\le &-(1-\tau)\left(1 - \frac{1}{\lambda^2}\right)\frac{1}{\mu_j^2}
		\end{eqnarray}
		by~\eqref{eq:taumuseries}.
		\item \textbf{Only $(r^0_j,r^s_j)$ is accurate.} This case can be analyzed the 
		same way as Case 2.1 to show that~\eqref{eq:phidecunsuccits} holds regardless 
		of whether the iteration is successful or unsuccessful.
		\item \textbf{Both $m_j$ and $(r^0_j,r^s_j)$ are inaccurate.} As in Case 1.4, 
		we have
		\[
			f(x_j+s_j) - f(x_j) \le \|J(x_j)^\top r(x_j)\| \|s_j\| + \frac{\revised{\nu}}{2}\|s_j\|^2.
		\]
		Using $\|J(x_j)^\top r(x_j)\| < \frac{\zeta}{\mu_j}$ and~\eqref{eq:stepsizebound}, 
		we obtain:
		\begin{eqnarray*}
			f(x_j+s_j) - f(x_j) 
			&\le &\zeta\frac{\|s_j\|}{\mu_j}+\frac{\revised{\nu}}{2}\|s_j\|^2 \\
			&\le &2\left(\zeta+\revised{\nu}\right)\frac{1}{\mu_j^2}.
		\end{eqnarray*}
		Then, applying Lemma~\ref{lemma:oursab} as in Case 1.4 yields
		\begin{equation*}
			\|r(x_j+s_j)\|^{1/2^\icc} - \|r(x_j)\|^{1/2^\icc} \le 
			\frac{4\left(\zeta+\revised{\nu}\right)}{\mu_j} 
			\frac{\|J(x_j)^\top r(x_j)\|}{\|r(x_j)\|^{\ipow}},
		\end{equation*}
		thus
		\begin{equation*} 
			\|r(x_j+s_j)\|^{1/2^\icc} - \|r(x_j)\|^{1/2^\icc} \le 
			\frac{4\zeta\left(1+\tfrac{\revised{\nu}}{\zeta}\right)}{\mu_j^2 \|r(x_j)\|^{\ipow}} 
			= \zeta C_3 \frac{1}{\mu_j^2 \|r(x_j)\|^{\ipow}}.
		\end{equation*}		
		and
		\begin{equation} \label{eq:case2d}
			\phi_{j+1} - \phi_j \; \le \; \left[ \tau C_3 \zeta + 
			(1-\tau)(\lambda^{2}-1)\right]\frac{1}{\mu_j^2\|r(x_j)\|^{\ipow}}.
		\end{equation}
	\end{enumerate}
	Combining all the subcases for Case 2, we can bound all of those 
	by~\eqref{eq:phidecunsuccits} save for Case 2.4, for which~\eqref{eq:case2d} 
	applies. Thus, we obtain:
	\begin{eqnarray*}
		& & \E{\Phi_{j+1}-\Phi_j | \mathcal{F}_{j-1}^{M \cdot R} \cap E^0_j \cap 
		\{\|\nabla f(X_j)\|  < \tfrac{\zeta}{\Mu_j}\}}  \\
		&\le & 
		[pq+p(1-q)+q(1-p)](1-\tau)\left(\frac{1}{\lambda^2}-1\right)
		\frac{1}{\Mu_j^2 \|r(X_j)\|^{\ipow}} \\
		& &
		+(1-p)(1-q)\left[ \tau C_3 \zeta + (1-\tau)(\lambda^{2}-1)\right]
		\frac{1}{\Mu_j^2 \|r(X_j)\|^{\ipow}} \\
		&\le & 
		-pq(1-\tau)\left(1 
		- \frac{1}{\lambda^2}\right)\frac{1}{\Mu_j^2 \|r(X_j)\|^{\ipow}} \\
		& &
		+(1-p)(1-q)\left[ \tau C_3 \zeta + (1-\tau)(\lambda^{2}-1)\right]
		\frac{1}{\Mu_j^2 \|r(X_j)\|^{\ipow}}.
	\end{eqnarray*}
	We now assume that $p$ and $q$ have been chosen such that $pq \ge \frac{1}{2}$ 
	and
	\begin{equation} \label{eq:condpqmuseries2}
		(1-p)(1-q) \le \frac{1}{4}\frac{(1-\tau)\left(1-\tfrac{1}{\lambda^2}\right)}
		{\tau C_3\zeta + (1-\tau)(\lambda^{2}-1)}
	\end{equation}
	holds. 
	Using~\eqref{eq:condpqmuseries2}, we obtain
	{\small{
	\begin{eqnarray} \label{eq:case2}
		\E{\Phi_{j+1}-\Phi_j | \mathcal{F}_{j-1}^{M \cdot R} \cap E^0_j \cap \{\|\nabla f(X_j)\| < 
		\tfrac{\zeta}{\Mu_j}\}} 
		&\le &-\frac{1}{4}(1-\tau)
		\left(1-\frac{1}{\lambda^2}\right)\frac{1}{\Mu_j^2 \|r(X_j)\|^{\ipow}},
	\end{eqnarray}
	}}
	which is the same amount of decrease as in~\eqref{eq:case1}. Letting 
	$\sigma :=\tfrac{1}{4}(1-\tau)\left(1-\tfrac{1}{\lambda^2}\right)$, we have
	then established that for every iteration $j$, 
	\[
		\E{\Phi_{j+1}-\Phi_j | \mathcal{F}_{j-1}^{M \cdot R} \cap E^0_j} 
		< -\frac{\sigma}{\Mu_j^2\|r(X_j)\|^{\ipow}}.
	\]
	As a result, the statement of the theorem holds.
\end{proof}


\end{appendices}

\end{document}